\documentclass[a4paper,11pt,reqno]{amsart}
\usepackage{amsfonts,amssymb,amsmath}
\usepackage[T1]{fontenc}
\usepackage[utf8]{inputenc}
\usepackage{dsfont}
\usepackage{amsmath,amssymb,latexsym,mathrsfs,amsthm}
\usepackage[usenames,dvipsnames,svgnames,table]{xcolor}
\usepackage{indentfirst}
\usepackage{amsmath}
\usepackage{amssymb}
\usepackage{hyperref}
\usepackage{enumitem}
\usepackage{comment}
\usepackage[hmargin=3.0cm,vmargin=3.0cm]{geometry} 
\usepackage{latexsym,mathrsfs,pb-diagram}
\usepackage{amssymb,amsmath,amsfonts,amsthm,mathtools}
\usepackage[usenames,dvipsnames,svgnames,table]{xcolor}
\usepackage[pdftex]{graphicx}

\numberwithin{equation}{section}

\theoremstyle{definition} 

\newtheorem{Def}{Definition}[section]

\theoremstyle{definition}

\newtheorem{Obs}[Def]{Remark}

\theoremstyle{plain}
\newtheorem{Pro}[Def]{Proposition}
\newtheorem{Cor}[Def]{Corollary}
\newtheorem{Teo}[Def]{Theorem}
\newtheorem{Lem}[Def]{Lemma}

\newcommand{\T}{\mathbb{T}}
\newcommand{\R}{\mathbb{R}}

\newcommand{\N}{\mathbb{N}}

\newcommand{\C}{\mathbb{C}}

\newcommand{\D}{\mathscr{D}}

\newcommand{\Z}{\mathbb{Z}}

\newcommand{\del}{\partial}

\renewcommand{\Re}{\mathrm{Re }}
\renewcommand{\Im}{\mathrm{Im }}

\newcommand{\ds}{\displaystyle}

\numberwithin{equation}{section}

\usepackage{color}

\definecolor{Red}{rgb}{1.00, 0.00, 0.00}
\definecolor{DarkGreen}{rgb}{0.00, 1.00, 0.00}
\definecolor{Blue}{rgb}{0.00, 0.00, 1.00}
\definecolor{Cyan}{rgb}{0.00, 1.00, 1.00}
\definecolor{Magenta}{rgb}{1.00, 0.00, 1.00}
\definecolor{DeepSkyBlue}{rgb}{0.00, 0.75, 1.00}
\definecolor{DarkGreen}{rgb}{0.00, 0.39, 0.00}
\definecolor{SpringGreen}{rgb}{0.00, 1.00, 0.50}
\definecolor{DarkOrange}{rgb}{1.00, 0.55, 0.00}
\definecolor{OrangeRed}{rgb}{1.00, 0.27, 0.00}
\definecolor{Black}{rgb}{0.00, 0.00, 0.00}
\definecolor{dark-magenta}{rgb}{.5,0,.5}
\definecolor{myblack}{rgb}{0,0,0}
\definecolor{DeepPink}{rgb}{1.00, 0.08, 0.57}
\definecolor{DarkViolet}{rgb}{0.58, 0.00, 0.82}
\definecolor{SaddleBrown}{rgb}{0.54, 0.27, 0.07}
\definecolor{darkgray}{gray}{0.5}
\definecolor{lightgray}{gray}{0.75}

\allowdisplaybreaks

\author{Alexandre Arias Junior}
\address{\small Departamento de Computa\c c\~ao e Matem\'atica, Universidade de S\~{a}o Paulo, Ribeirão Preto, SP, 14040-900, Brazil}
\email{alexandre.ariasjunior@usp.br}
\thanks{The first author was partially supported by CAPES, grant 88882.381844/2019-01 and partially supported by FAPESP, grant 2022/01712-3}

\author{Bruno de Lessa Victor}
\address{\small Departamento de Matem\'atica, Universidade Federal de Santa Catarina, Florianópolis, SC, 88040-900, Brazil}
\email[Corresponding Author]{bruno.lessa@ufsc.br}
\thanks{The second author was supported by FAPESP, grant 2021/03199-9.}


\title[]{The Cauchy problem for a class of linear degenerate evolution equations on the torus}

\begin{document}

\begin{abstract}
	We study, in the periodic setting, the well-posedness of the Cauchy problem associated to the operator $P(t, D_{x}, D_{t}) = D_{t} - a_{2}(t) \Delta_{x} + \ds \sum_{j = 1}^{N} a_{1, j}(t) D_{x_{j}} + a_{0}(t)$, with $T> 0$, $t \in [0, T]$ and $a_{2}, a_{1,1}, \ldots, a_{1, N}, a_{0} \in C \left([0, T]; \C \right)$. Using Fourier analysis techniques, we obtain a complete characterization for the well-posedness of a class of degenerate initial-value problems in the Sobolev, Smooth, Gevrey and Real-Analytic frameworks. 
\end{abstract}

	\maketitle
	
	\section{Introduction}
Linear evolution equations have attracted the attention of a great number of mathematicians for several decades, with some particular problems being studied as far as two hundred years ago. This is for instance the case of the Heat Equation, investigated by Jean-Baptiste Fourier in his book \emph{Théorie analytique de la chaleur (1822)}, where the concept of \emph{Fourier series} was first introduced and the cornerstones of what became known later as \emph{Fourier Analysis} were established.

Proceeding to the 20th and 21st centuries, one can find a huge number  of works concerning evolution problems (\emph{see for instance \cite{ aac1,aac2,abz1,abz2,ac,cc, ma} and references therein}); they may be concerned with different functional settings such as Sobolev, $C^{\infty}$ and  Gevrey ,  but most of them deal with domains in the form $I \times \R^{N}$, where $I$ is a subinterval of $[0, + \infty)$. In the particular case of second order operators, even the most classic examples of Cauchy problems  do not admit well-posedness in the most classic settings, such as $C^{\infty}(\R^{N})$ or $G^{s}(\R^{N})$ (see \cite{m}). Thus authors usually look for such properties  in spaces like  $H^{\infty}(\R^{N})$ and $H_{s}^{\infty}(\R^{N})$, whose elements decay at infinity. 

Consider for example the following class of second order evolution operators:
\begin{equation}\label{eq_beluga}
	P(t,x,D_t,D_x) = D_t - a(t,x)\Delta_{x} + \sum_{j=1}^{N}a_j(t,x)D_{x_j} + a_0(t,x), \quad t \in [0,T], x \in \R^{N},
\end{equation}
where $D_{x_j} = -i\partial_{x_j}$, $\Delta_x = \sum_{j=1}^{N} \partial^{2}_{x_j}$ and the coefficients are assumed to be complex valued, continuous with respect to time and $\mathcal{B}^{\infty}(\R^{N})$-regular in the space $x$. Here $\mathcal{B}^{\infty}(\R^{N})$ stands for the space of all smooth functions which are not only bounded, but this also holds for each of their derivatives. 

If we split the leading coefficient into its real and imaginary part: $a(t,x) = b(t,x) + ic(t,x)$, the class given in \eqref{eq_beluga} englobes two special cases. When $c(t,x) \leq -\varepsilon$ for every $t$ and $x$, $P$ is a typical example of parabolic operator. In this situation  one obtains in general well-posedness in $L^{2}(\R^{N})$ without any loss of derivatives. This is the case for instance in \cite{kg} and \cite{t}, where a fundamental solution is constructed. On the other hand, when $c(t,x) \equiv 0$, \eqref{eq_beluga} is a Schr\"odinger type operator, a typical example of non-kowalewskian operator that is not parabolic. The well-posedness in this situation turns out to be \emph{way more intricate}, being usually dictated by the imaginary part of the first order terms; for instance \cite{i2, i3} and \cite{m} exhibit necessary and sufficient conditions for $L^{2}(\R^N)$ and $H^{\infty}(\R^N)$ well-posedness of the related problem. Moreover, in \cite{dre} the author presents a necessary condition in the \emph{Gevrey} framework. Frequently these conditions are associated to the \emph{decay} of the imaginary part of the first order coefficients, especially when one searches sufficient conditions for well-posedness in the \emph{Gevrey setting} (see for example \cite{cr, kb}.

Now observe that, if one replaces $\R^{N}$ by $\T^{N}$ in any of the problemas above, the concept of \emph{decay at infinity} does not fit anymore. And there is no guarantee that even by making the \emph{proper translations} between both environments (such as $H^{r}(\R^{N}) \leftrightarrow H^{r}(\T^{N})$ and  $H^{\infty}(\R^{N}) \leftrightarrow C^{\infty}(\T^{N})$) one would obtain analogous results. In fact, it is widely known the existence of  disparities between properties for operators that act both in the euclidean and periodic settings,  which makes our inquiries much more compelling.  We recall for instance the existence of constant coefficient operators which are globally hypoelliptic when acting on periodic distributions, but are not hypoelliptic in the local sense (see \cite{gw}).

An extensive search was made in the literature, but we were not able to spot any manuscript that deals with any problem of the nature just described. In the parabolic case, we have reasons to believe that one has similar results (Corollary \ref{Lincoln} is a first confirmation of our educated guess). Similarly, we think that the Schr\"odinger case can be completely solved. The authors are currently working with both problems.

Due to the fact that there were no references available, we started with a problem for which Fourier analysis tools could be applied (i.e the  coefficients of our operator do not depend on space variables).  Even in this case, it becomes clear that although some of our techniques could be reproduced for the euclidean case, several results rely strongly on the fact that we are working in the periodic setting, particularly in Sections \ref{The Wrestler}, \ref{Walk the Line} and \ref{The Dark Knight}.

Now we proceed to the content of our work; let  $\T^N = \R^N / 2 \pi \Z^N$ be the $N$-dimensional torus, for some $N \in \N$, and fix a positive real number $T$. We deal with  the  following  class of operators:
\begin{equation} \label{Her} 
	P(t, D_{x}, D_{t}) = D_{t} - Q(t, D_{x}) = D_{t} - a_{2}(t) \Delta_{x} + \ds \sum_{j = 1}^{N} a_{1, j}(t) D_{x_{j}} + a_{0}(t), \ \ \ \ t \in [0, T], \ x \in \T^N,
\end{equation}	
where $\Delta_{x}$ denotes the Laplace Operator, $D_{t} = \ds \frac{1}{i} \ds \frac{\partial}{\partial t}$, $ D_{x_{j}} = \ds \frac{1}{i} \ds  \frac{\partial}{\partial x_{j}}$ and $a_{0}(t), a_{1,1}(t), \ldots, a_{1, N}(t), a_{2}(t)$ are all elements of $C \left([0, T]; \C \right)$. The main purpose of the present work is to study well-posedness for  initial-value problems in the form 
\begin{equation} \label{Schindler's List}
	\left\{ \begin{array} {rl}
		Pu(t,x) &= f(t,x), \\
		u(0, x) & = g(x), 
	\end{array} \right. \ \ \forall t \in [0, T], \ \forall x \in \T^{N},
\end{equation}
in the framework of Sobolev ($H^{r}(\T^{N})$, $r \in \R$), smooth ($C^{\infty}(\T^{N}))$, Gevrey ($G^{s}(\T^{N})$, $s > 1$) and real-analytic ($C^{\omega}(\T^{N})$) spaces.  By well-posedness we mean the existence and uniqueness of a $u(t, x)$ in $C \left([0, T]; X \right)$ which solves the Cauchy problem, with $X$ representing any of the aforementioned spaces. 
  
As it should be expected, the most important factor to the well-posedness of \eqref{Schindler's List} is the sign of the imaginary part of $a_{2}$. Indeed, generally speaking, we prove the following statements: 
\begin{itemize} [leftmargin=*]
	\item If there exists  $t^{\star} \in [0, T]$ such that $\Im (a_{2}(t^{\star})) > 0$, then  \eqref{Schindler's List} is ill-posed in each of the aforementioned spaces (see Theorem \ref{American  Beauty}).
	\item When $\Im (a_{2}) < 0$, \eqref{Schindler's List} is well-posed for each framework (see Theorem \ref{Rear Window} and Corollary \ref{Lincoln}).
\end{itemize}

It is important to emphasize that both results \emph{were expected}; a  similar fact was proved in \cite{p}, in the euclidean environment. Nevertheless they were still included in this manuscript since they are \emph{straightforward consequences} of results we consider quite interesting, such as Theorem \ref{Cries and Whispers}  and Proposition  \ref{The Graduate}. 

Hence the only situation left uncovered is the one where $\Im (a_{2}) \leq 0$ and vanishes at some point. This is by far the most interesting case and now the behavior of the imaginary part of the first-coefficients also must be accounted.
First we deal with another case which should be more predictable, when $\Im (a_{2}) \equiv 0$: 
\begin{itemize} [leftmargin=*]
	\item \eqref{Schindler's List} is well-posed in any of the settings if  $\Im (a_{1, j}) \equiv 0$ for each $j \in \left\{1, \ldots, N \right\}$ --- a particular case of Theorem  \ref{First Man} and Corollary \ref{Bourne Identity}.
	\item \eqref{Schindler's List} is ill-posed in each framework if  $\Im (a_{1, j}) \not\equiv 0$ for some $j \in \left\{1, \ldots, N \right\}$ (see Theorem  \ref{La La Land}).
\end{itemize}

Finally, we move towards the most important and compelling part of the article (Section \eqref{The Dark Knight}): the \emph{degenerate case}.  Assuming that $\Im(a_{2})$ is never strictly positive, has a finite number of zeros $\left\{t_{1}, \ldots, t_{m} \right\}$ and vanishes to a finite order at each of them, we characterize well-posedness by comparing the order of vanishing of each $\Im (a_{1,j})$ at $t_{k}$, $k \in \left\{1, \ldots, m \right\}$.  Since our main statement  can be quite complicated to absorb at a first glance, we present an easier example. Let 
\begin{equation*}
P(t, D_{x}, D_{t}) = D_{t} + it^{k} \ds \frac{\del^{2}}{\del x^{2}} +  i t^{\ell} D_{x} + a_{0}(t), \ \ \ \ t \in [0, T], \ x \in \T, \ \ \ k, \ell \geq 0. 
\end{equation*}
As a consequence of Theorem \ref{Tropa de Elite II} and Corollary \ref{The Bourne Ultimatum} one obtains: 
\begin{itemize}[leftmargin=*]
	\item \emph{ If $k \leq 2 \ell + 1$, then \eqref{Schindler's List} is well-posed in each of the settings;} 
	\item \emph{ When $k > 2 \ell + 1$, the problem is ill-posed in both Sobolev and $C^{\infty}$ frameworks. Moreover, it will be well-posed in $G^{s}$ if and only if}
	\begin{equation*}
	s < \ds \frac{k - \ell}{k - 2 \ell - 1}.
	\end{equation*}
\end{itemize}

Since $C^{\omega}(\T^{N}) = G^{1}(\T^{N})$, it is worth noting  that in the example above  one always has well-posedness in $C^{\omega}(\T^{N})$. This in fact holds true for the much more general statement made in Theorem \ref{Tropa de Elite II}.  We prove in Section \ref{Ex Machina} that such phenomenon is directly related to \eqref{Her} being a second-order operator with respect to the space variables. Furthermore, it is important to mention that when any of the coefficients vanishes to an \emph{infinite order} at some point, it is \emph{not possible} to obtain such a characterization anymore; this is the content of Remark \ref{The Super Mario Bros. Movie}.

This manuscript is organized in the following manner: Section \ref{The Post} is devoted for the precise definitions of the spaces for which we are interested to investigate well-posedness and its properties regarding Fourier series. We also provide an explicit formula for the Fourier coefficients of the formal solution of the inital-value problem \eqref{Schindler's List}. The subsequent section contains results that allow us to conclude the following:  the influence of the imaginary parts of the coefficients of \eqref{Her} to the well-posedness of \eqref{Schindler's List}  is significantly stronger when compared with the real parts (Theorem \ref{Cries and Whispers}). 

In Section \ref{Walk the Line} we show that the solution of \eqref{Schindler's List} possesses some regularity provided that an \emph{a priori} energy estimate in the phase space holds (Lemma \ref{American Gangster}). Furthermore, we prove (Proposition \ref{The Graduate}) that
\begin{align*}
&\text{$H^{r}$ well-posedness is stronger than $C^{\infty}$ well-posedness}; \\
&\text{$C^{\infty}$ well-posedness is stronger than $G^{s}$ well-posedness (for any $s > 1$)}; \\
&\text{$G^{s}$ well-posedness is stronger than $C^{\omega}$ well-posedness}.
\end{align*}
As we already observed, the immaginary part of the leading coefficient $a_2$ plays an important role on the well-posedness of \eqref{Schindler's List}. Roughly speaking, $\Im~a_2 \leq 0$ is a necessary condition for well-posedness (Theorem \ref{The Lord of the Rings: The Return of the King}), whilst $\Im a_2 < 0$ (that is, $P$ given by \eqref{Her} is parabolic) is a sufficient one (Theorem \ref{Rear Window} and Corollary \ref{Lincoln}). These are briefly the contents of Sections  \ref{Memento} and \ref{Gladiator} respectively. 

Section \ref{The Dark Knight} contains the main result of this work, namely Theorem \ref{Tropa de Elite II}. We prove energy estimates in the phase space for a class of degenerate operators, yielding  well-posedness and ill-posedness for the associated Cauchy problem, depending on the behavior of the zeros of the coefficients. Finally, we close the paper making some final remarks and examples at Section \ref{Ex Machina}.
 
\section{Preliminaries} \label{The Post}

Before we  deal with any particular case of  \eqref{Schindler's List}, it is essential first to define precisely the objects which we alluded to at the introduction. As usual, we denote $C^{\infty}(\T^{N})$ as the space of complex-valued smooth functions defined on the torus and by $\D'(\T^N)$ its topological dual space, endowed with the topology of  \emph{pointwise convergency}.

Given  $u \in \D'(\T^{N})$, one defines for each $\xi \in \Z^N$ its \emph{Fourier coefficient} 	$\widehat{u}(\xi) = \left\langle u, e^{-ix\xi} \right\rangle.$ It is possible to prove that $u$ is completely described by its Fourier series. That is, 
\begin{equation*}
	u = \ds \lim_{j \to +\infty} \ds \sum_{|\xi| \leq j}  \widehat{u}(\xi) e^{i\xi x }  = \ds \sum_{\xi \in \Z^{N}} \widehat{u}(\xi) e^{i\xi x },
\end{equation*}
with convergence in $\D'(\T^N)$. For any fixed $r \in \R$,  consider  $\left\| \cdot \right\|_{r}$ the application described by
	\begin{equation*}
		\D'(\T^N) \ni u \mapsto \left\| u \right\|_{H^r(\T^N)}^{2}: =  \ds \sum _{\xi \in \Z^N} |\widehat{u}(\xi)|^{2} (1+ |\xi|)^{2r}.
	\end{equation*}
	\emph{We say that $u \in H^{r}(\T^N)$ if $\left\| u \right\|_{H^{r}(\T^N)} < + \infty$}. The spaces $H^{r}(\T^N)$ are usually known as \emph{periodic Sobolev spaces of order} $r$; they are Hilbert spaces, with inner product given by $(u, v)_{H^{r}(\T^N)} = \ds \sum_{\xi \in \Z^{N}} \widehat{u}(\xi) \overline{\widehat{v}(\xi)} (1+|\xi|)^{2r}.$ Moreover, $H^{r_{1}}(\T^N) \subset H^{r_{2}}(\T^N)$ if and only if $r_{1} \geq r_{2}$ and 
\begin{equation*}
C^{\infty}(\T^N) = \ds\bigcap_{r \in \R} H^{r}(\T^N), \ \ \ \ \D'(\T^N) = \ds \bigcup_{r \in \R} H^{r}(\T^N).
\end{equation*}

The topology of $C^{\infty}(\T^{N})$ is given by the projective limit of the Sobolev spaces (For more details on projective and injective limits of locally convex spaces, we recommend \cite{kom} or the \textit{Appendix A} in \cite{mor}). That is, 
\begin{equation} \label{The Godfather}
	C^{\infty}(\T^N) = \ds \lim_{\stackrel{\longleftarrow}{\tau \in \Z_{+}}} H^{\tau}(\T^N).
\end{equation}
As a consequence of  Rellich's Theorem, the Sobolev inclusions  are \emph{compact}, which implies $C^{\infty}(\T^{N})$ is Fréchet-Schwartz. In particular, it is Montel, reflexive and separable.  Next we present the space of Gevrey functions. 

\begin{Def}
	For any $s \geq 1$, we say that $f \in C^{\infty}(\T^N)$ is an element of $G^{s}(\T^{N})$ if there exist $C, h > 0$ such that
	\begin{equation} \label{One Flew Over the Cuckoo's Nest}
		\left|D^{\alpha} f (x) \right| \leq C h^{|\alpha|} |\alpha|!^{s}, \ \ \forall \alpha \in \Z_{+}^{N}, \ \ \forall x \in \T^N. 
	\end{equation}
\end{Def}

The spaces $G^{s}(\T^N)$ are usually known as \emph{$2\pi$-periodic Gevrey spaces of order $s$}. When $s = 1$ we obtain the spaces of real-analytic $2 \pi$-periodic functions, also denoted by $C^{\omega}(\T^N)$. Let us briefly describe how the usual topology imposed in $G^{s}(\T^N)$ is constructed (we recommend \cite{lv} for more details and proofs of the results stated). For any $h> 0$, we set
\small{
\begin{equation*}
	G_{ h}^{s}(\T^N) = \left\{f \in G^{s}(\T^N); \ \left\|f \right\|_{h,s} := \ds \sup_{\alpha \in \Z_{+}^{N}} \left(\ds \frac{\left\|D^{\alpha} f  \right\|_{\infty}}{h^{|\alpha|}  |\alpha|!^{s}} \right)  < \infty \right\}. 
\end{equation*}
}
\normalsize
Clearly $\left\| \cdot \right\|_{h, s}$ is a norm and in fact it can be proved that $\left(G_{ h}^{s}(\T^N), \left\| \cdot  \right\|_{h, s} \right)$ is a Banach space. Furthermore if $0 \leq h_{1} \leq h_{2}$, the inclusions $G_{h_{1}}^{s}(\T^N) \hookrightarrow G_{h_{2}}^{s}(\T^N)$ are \emph{compact}. Taking any strictly increasing sequence $\left\{h_{n}\right\}_{n \in \N}$ of positive numbers such that $ h_{n} \to + \infty$ and setting
\begin{equation*}
	G^{s}(\T^N) = \ds \lim_{\stackrel{\longrightarrow}{n \in \N}} G_{h_{n}}^{s}(\T^N),
\end{equation*}
we conclude that $G^{s}(\T^N)$ is a \emph{DFS space} (for more details, we recommend Section $5$ in Appendix $A$ of \cite{mor}). In particular, we obtain from Corollary $A.5.11$ in \cite{mor} that $G^{s}(\T^N)$ is a Montel and reflexive space. Finally, it is important to mention that it is  possible to characterize elements of $G^{s}(\T^N)$ in terms of its Fourier Series. 
\begin{Pro} \label{Dr. Strangelove} 
Let $u \in \D'(\T^N)$; then $u \in G^{s}(\T^N)$ if and only if there exist $C, \delta > 0$ such that
\begin{equation} \label{Amadeus}
\left| \widehat{u}(\xi) \right| \leq C e^{- \delta |\xi|^{1/s}}, \ \ \forall \xi \in \Z^{N}. 
\end{equation}
Moreover, if $u \in G^{s}(\T^N)$ we can write $u = \ds \sum_{\xi \in \Z^{n}} \widehat{u}(\xi) e^{i x \xi}$, with convergence in $G^{s}(\T^N)$. 
\end{Pro}
We point out the following relation between the numbers $\delta$ in \eqref{Amadeus} and $h$ in \eqref{One Flew Over the Cuckoo's Nest}. 
\begin{Lem} \label{The People vs. Larry Flynt}
	Given arbitrary $s \geq 1$ and $h > 0$, there exists $\lambda (N, s) >0$ such that 
	\begin{equation*}
		f \in G_{h}^{s}(\T^N) \ \Rightarrow \ |\widehat{f}(\xi)|  \leq  2^{s} \left\|f \right\|_{h, s} e^{- \frac{\lambda}{h^{1/s}} |\xi|^{1/s}}, \ \ \forall \xi \in \Z^{N}.
	\end{equation*}
\end{Lem}

\begin{Lem} \label{Boyhood}
	Fix $s\geq 1$ and let $g \in G^{s}(\T^N)$ such that $|\widehat{g}(\xi)| \leq C_{1} e^{- \delta |\xi|^{1/s}}$, for any $\xi \in \Z^N$ and   $C_{1}, \delta > 0$. Set $h = \left(\frac{2s} {\delta} \right)^{s}$;  then $g \in G^{s}_{h} (\T^N)$ and there exists a positive constant $C_{2}(\delta, N, s)$ such that $\left\|g \right\|_{h, s} \leq C_{1} C_{2}$. 
\end{Lem}

\subsection{Characterizations for  classes of continuous curves} \label{Molly's Game}

The main goal in this subsection is to completely understand the following spaces of functions:
\begin{enumerate} [leftmargin=*,   label=\roman*)]
	\item $C \left([0, T]; H^{r}(\T^{N}) \right)$, for any  $r \in \R$;  \label{Star Wars}
	\item $C \left([0, T]; C^{\infty}(\T^{N}) \right);$ \label{The Empire Strikes Back}  
	\item $C \left([0, T]; G^{s}(\T^{N}) \right)$, for any $s \geq 1$. \label{Return of the Jedi}
\end{enumerate}
Since we intend to work extensively with them throughout this manuscript, it is important to understand the meaning and different characterizations of their elements.

Note that case \ref{Star Wars} is described by the continuity of applications between metric spaces. Regarding case \ref{The Empire Strikes Back}, it is a consequence of a standard result (see for instance Theorem $A.3.3$ in \cite{mor}) the following 
\begin{Teo} \label{The Godfather II}
	An application $g: [0, T] \to C^{\infty}(\T^N)$ is continuous if and only if for every $r \in \R$ the map $\iota_{r} \circ g: [0,T] \to H^{r}(\T^N)$ is continuous, where  $\iota_{r}: C^{\infty}(\T^N) \to H^{r}(\T^N)$  denotes the canonical inclusion.  
\end{Teo}

We now proceed to case \ref{Return of the Jedi}; let $u: [0, T] \to G^{s}(\T^N)$ be a continuous map; since $[0, T]$ is compact, the same holds true for $u([0,T])$. In particular, the set is bounded, which means that
\begin{equation*}
	\text{$u([0, T])$ is a bounded subset of  $G^{s}_{h}(\T^N)$}, \ \text{for some  $h > 0$ (see Theorem A.5.7 in \cite{mor})}.  
\end{equation*}
By possibly increasing $h$ and applying Theorem $6'$ in \cite{kom}, we may restrict the codomain and obtain $u: [0, T] \to G^{s}_{h}(\T^N)$ a continuous map. Therefore we have the following
\begin{Teo} \label{On the Waterfront}
	A map $u: [0, T] \to G^{s}(\T^N)$ is continuous if and only if there exists $h > 0$ such that $u([0, T]) \subset G^{s}_{h}(\T^N)$ and $u: [0, T] \to G^{s}_{h}(\T^N)$ is continuous.  
\end{Teo}

\subsection[Distributions with values in Sobolev Spaces]{Distributions in $(0,T)$ with values in $H^{r}(\T^{N})$}
	
As stated previously, we would like to find solutions for \eqref{Schindler's List} in the form $u(t,x)$, where $u (\cdot , x)$ is continuous for any fixed $x$ and $u(t, \cdot)$ is (at least) an element of $H^{r}(\T^{N})$ for any $t$ fixed. A quite intuitive way to reach this goal would be to write $u$ in terms of its Fourier series and to convert the problem into a denumerable family of first-order ODE's parametrized by $\xi \in \Z^{N}$. In order to turn such idea into a formal result, we introduce  the following concept:

\begin{Def} \label{Barry Lyndon} 
Given $T > 0$ and $r \in \R$, we define  $D'((0,T);H^{r}(\T^{N}))$ as the space of \emph{continuous linear functionals} $u : C^{\infty}_{c}((0,T)) \to H^{r}(\T^{N}).$
\end{Def}

Similarly to other spaces of distributions, we impose in $D'((0,T); H^{r}(\T^{N}))$ the topology of \emph{pointwise convergency}. Thus a sequence $\left\{u_{k}\right\}_{k \in \N}$ converges to $u$ in $D'((0,T);H^{r}(\T^{N}))$ if and only if
\begin{equation*}
	\| u_{k} (\phi)  - u(\phi) \|_{H^{r}(\T^{N})} \to 0, \ \ \ \forall \phi \in C^{\infty}_{c}\left((0,T) \right).
\end{equation*}
Next we relate this unusual space of distributions with the sets described in  Subsection \ref{Molly's Game}.  
	
	\begin{Pro} \label{Raging Bull} 
		$C([0,T]; H^{r}(\T^{N}))$ is a subspace of $D'((0,T);H^{r}(\T^{N}))$.
	\end{Pro}
	
	\begin{proof}
		Take $v \in C([0,T]; H^{r}(\T^N))$; then for every $\phi(t) \in C^{\infty}_{c}((0,T))$, the function $w(t) = \phi(t) v(t)$ is an element of  $C([0,T]; H^{r}(\T^N))$. Now consider
		\begin{equation*}
			\begin{split}
				\mathcal{T}: C([0,T]: H^{r}(\T^N)) &\to D'((0,T); H^{r}(\T^N)) \\
				u &\mapsto \mathcal{T}_u: C^{\infty}_{c}((0,T)) \to H^{r}(\T^N),	
			\end{split}
		\end{equation*}
		where $\mathcal{T}_u$ denotes the map $\phi \mapsto \ds \int_{0}^{T} \phi(t)u(t) dt =: \langle \mathcal{T}_u, \phi \rangle$
		and the integral above represents the \emph{Bochner Integral taking values in $H^{r}(\T^N)$}(for more details we recommend Section $5$ of Chapter $5$ in \cite{yos}). 
		
		Since $\phi(t)u(t) \in C([0,T]; H^{r}(\T^N))$ it follows that
		\begin{align*}
			\left\|\langle \mathcal{T}_u, \phi \rangle \right\|_{H^{r}(\T^N)} \leq  \int_{0}^{T} \| \phi(t)u(t) \|_{H^{r}(\T^N)} dt \leq T \left(\ds \max_{t \in [0,T]} \|u(t)\|_{H^{r}(\T^N)} \right) \left( \ds \max_{t \in [0,T]} |\phi(t)| \right),
		\end{align*}
		which shows that $\mathcal{T}_{u} \in D'((0,T); H^{r}(\T^{N})) $ and $\mathcal{T}$ is continuous.
		
		Finally we  prove that $\mathcal{T}$ is injective. For this purpose, take $u \in C([0,T]; H^{r}(\T^{N}))$ satisfying $\langle T_u, \phi \rangle = 0$ for every $\phi \in C^{\infty}_{c}((0,T))$ and let $\iota \in (H^{r}(\T^{N}))^{'}$ be an arbitrary continuous linear functional. Then (see Corollary $2$, pg 134 of \cite{yos})
		\begin{equation*}
			0 = \iota (\langle T_u, \phi \rangle) = \iota \left( \int_{0}^{T} \phi(t) u(t) dt \right) = \int_{0}^{T} \iota (\phi(t)u(t)) dt = \int_{0}^{T} \phi(t) \iota(u(t)) dt, \ \ \forall \phi \in C^{\infty}_{c}((0,T)).
		\end{equation*}
		Since $\iota(u(t))$ is a continuous scalar function, it follows that $\iota(u(t)) = 0$ for every $t \in (0,T)$. Hence $u(t) = 0$ for every $t \in [0,T]$, which finalizes the proof. 
	\end{proof}
	
	\begin{Obs}
		From now on we will denote $T_{u}$ just as $u$. 
	\end{Obs}

	Using the spaces introduced in Definition \ref{Barry Lyndon}, one can establish the concept of weak derivative (with respect to $t$) for elements of $ C([0,T]; H^{r}(\T^N))$. Indeed, set for any $u \in D'((0,T); H^{r}(\T^{N}))$
	\begin{equation*}
		\langle D_t u, \phi \rangle = \langle u, - D_t\phi \rangle, \ \ \ \forall \phi(t) \in C^{\infty}_{c} \left((0,T) \right).
	\end{equation*}
	In this case, it is not difficult to check that $D_{t}:D'((0,T); H^{r}(\T^{N})) \to  D'((0,T); H^{r}(\T^{N}))$ is continuous.

	\subsection[Fourier series for this space of distributions]{Fourier series for elements of $C([0,T]; H^{r}(\T^{N}))$} \label{Gravity}

	Note that we still have not used  any property of the spaces $H^{r}(\T^N)$ other than they are Banach. The fact that we may write any element of it in terms of its Fourier series will be extremely useful throughout the whole section.
	
	Let $u \in C \left([0, T]; H^{r}(\T^N) \right)$; then each fixed $t \in [0, T]$ we have $u(t,x) = \ds \sum_{\xi \in \Z^N} \widehat{u}(t,\xi) e^{i\xi x}$, where $\widehat{u}(t,\xi) = \langle u(t,x) , e^{-i\xi x} \rangle$. It follows from standard properties of Fourier analysis that such series is \emph{uniquely defined}. As a consequence of the continuity of $u$, for each $\xi \in \Z^N$ fixed the map $t \mapsto \widehat{u}(t,\xi)$ is continuous. Moreover, it also implies that 
	\begin{equation*}
		\begin{split}
			[0, T] &\to H^{\tau}(\T^N) \\
			t &\mapsto \widehat{u}(t,\xi)e^{i\xi x}
		\end{split}
	\end{equation*}
	is an element of $C([0,T]; H^{\tau}(\T^N))$, for any $\tau \in \R$. 
	
	\begin{Teo} \label{Chinatown} 
		Let $u \in C([0,T]; H^{r}(\T^N))$; then the limit 
		\begin{equation*}
			v := \ds \sum_{\xi \in \Z^N} \widehat{u}(t,\xi) e^{i\xi x}
		\end{equation*}
		is an element of  $D'((0,T); H^{r}(\T^N))$ and $u = v$. In other words, $u$ can be written in terms of its partial Fourier Series in  $D'((0,T); H^{r}(\T^N))$. 
	\end{Teo}
	\begin{proof}
		We fix $u \in C([0,T]; H^{r}(\T^N))$ and set, for any $k\in \N_{0}$,
		\begin{equation*}
			s_k(t, x) = \sum_{|\xi| \leq k} \widehat{u}(t,\xi) e^{i\xi x}, \ \ \ \forall t \in [0,T], \ \ \forall x \in \T^N.
		\end{equation*} 
		It follows from our last claim that $s_k(t) \in C([0,T]; H^{r}(\T^N))$ and from Proposition \ref{Raging Bull} that 
		\begin{equation*}
			\langle s_{k}(t, x), \phi(t) \rangle = \ds \int_{0}^{T} \phi(t) s_{k}(t, x) dt, \ \ \forall \phi \in C_{c}^{\infty}\left((0, T) \right).
		\end{equation*}
		
		Fix $\phi \in C_{c}^{\infty}\left((0, T) \right)$; then the sequence $\phi(t) s_k(t, x)$ is uniformly bounded in $H^{r}(\T^N)$. Indeed,  
		\begin{align}
			\left\| \phi(t) s_k(t, x) \right\|_{H^{r}(\T^N)} &\leq \left\| \phi(t) \left(s_k(t, x) - u(t) \right) \right\|_{H^{r}(\T^N)} + \left\| \phi(t)  u(t) \right\|_{H^{r}(\T^N)} \nonumber \\
			&\leq \left(\ds \max_{t \in [0,T]} |\phi(t)| \right) \left[\left\| \left(s_k(t) - u(t) \right) \right\|_{H^{r}(\T^N)} + \left\| u(t) \right\|_{H^{r}(\T^N)} \right] \label{Blade Runner}.
		\end{align} 
		On the other hand, for any fixed $t \in [0, T]$ 
		\begin{equation} \label{Blade Runner 2049}
			\| s_k(t, x) - u(t) \|^{2}_{H^{r}(\T^N)} = \ds \sum_{|\xi| > k} |\widehat{u}(t,\xi)|^{2} (1+ |\xi|)^{2r} \leq \left\| u(t) \right\|_{H^{r}(\T^N)}^{2}.
		\end{equation}
		By associating \eqref{Blade Runner} to \eqref{Blade Runner 2049} we deduce that 
		\begin{equation*}
			\ds \max_{t \in [0,T]} \left\| \phi(t) s_k(t,x) \right\|_{H^{r}(\T^N)} \leq 2 \left(\ds \max_{t \in [0,T]} |\phi(t)| \right) \left(\ds \max_{t \in [0,T]} \left\| u(t) \right\|_{H^{r}(\T^N)} \right),
		\end{equation*}
		which confirms our assertion. 
		
		Finally, it follows from the identity in \eqref{Blade Runner 2049} that $\ds \lim_{k \to \infty} \phi(t) s_k(t) =  \phi(t)u(t)$ in $H^{r}(\T^N)$, for every $t \in [0, T]$. By the Dominated Convergence Theorem, we obtain
		\begin{equation*}
			\left\| \int_{0}^{T} \phi(t)\{s_k(t)-u(t)\} dt \right\|_{H^{r}(\T^N)} \leq   \int_{0}^{T} \|\phi(t) s_k(t)- \phi(t)u(t)\|_{H^{r}(\T^N)}  dt \to 0 \ \text{as} \ k \to \infty.
		\end{equation*}
		Therefore
		\begin{equation*}
			\lim_{k \to \infty} \langle s_k(t, x), \phi(t) \rangle = \lim_{k \to \infty} \int_{0}^{T} \phi(t) s_k(t, x)  dt 
			= \int_{0}^{T} \phi(t) u(t) dt = \langle u(t), \phi(t) \rangle,
		\end{equation*}
		which finalizes the proof of the theorem.   
	\end{proof}

	\begin{Obs} \label{The Pianist} 
		Since $D_t$ acts continuously in $D'((0,T); H^{r}(\T^N))$  for any $u \in C \left([0, T]; H^{r}(\T^N) \right)$  
		\begin{equation*}
			D_t u = D_t \sum_{\xi \in \Z^N} \widehat{u}(t,\xi) e^{i\xi x} = \sum_{\xi \in \Z^N} D_t \widehat{u}(t,\xi) e^{i\xi x}.
		\end{equation*}	
		That is, we are able to write the weak derivative of $u(t)$ as the limit  of 
		\begin{equation*}
			\sum_{|\xi| \leq k} D_t \widehat{u}(t,\xi) e^{i\xi x},
		\end{equation*}
		where  $D_t\widehat{u}(t,\xi) e^{i\xi x}$ denotes the following element of $D'(0,T);H^{r}(\T^N))$:
		\begin{equation*}
			\langle D_t \widehat{u}(t,\xi) e^{i\xi x}, \phi(t) \rangle = \left[\int_{0}^{T} \widehat{u}(t,\xi)(-D_t\phi(t)) dt \right] e^{i\xi x}, \ \ \forall \phi \in C^{\infty}_{c}\left((0, T) \right).
		\end{equation*}
		
	\end{Obs}
	
	\begin{Obs} \label{Munich}
		Let $u \in C \left([0, T]; H^{r}(\T^N) \right)$; for any $\alpha \in \Z_{+}^{N}$ it follows from the continuity of derivatives in Sobolev spaces that $D_{x}^{\alpha} u \in C \left([0, T]; H^{r - |\alpha|}(\T^N) \right)$. By Theorem \ref{Chinatown} we have
		\begin{align*}
			D_{x}^{\alpha} u &=  \ds \sum_{\xi \in \Z^{N}}  \widehat{D_{x}^{\alpha} u}(t, \xi) e^{i\xi x} = \ds \sum_{\xi \in \Z^{N}} \widehat{u}(t,\xi) \xi^{\alpha} e^{i\xi x}, \ \ \ \text{with convergence in} \ D'((0,T); H^{r-|\alpha|}(\T^N)).
		\end{align*}
	\end{Obs}

	\subsection{A family of ordinary differential equations} \label{Apollo 13}

We now return to the original problem; suppose that for some $ \tau \in \R$, $f \in C \left([0,T]; H^{\tau}(\T^N) \right)$ and  $g \in H^{\tau}(\T^N)$ we are able to find $u \in C \left([0,T]; H^{\rho}(\T^N) \right)$ which solves \eqref{Schindler's List}, for some $\rho \in \R$. Then 
	\begin{equation*}
		D_{t} u =  a_{2}(t) \Delta_{x}u - \ds \sum_{j = 1}^{N} a_{1, j}(t) D_{x_{j}}u - a_{0}(t)u + f(t, x).
	\end{equation*}
	If $\sigma = \min \left\{\tau, \rho - 2 \right\}$, note that the right-hand side is an element of $C \left([0, T]; H^{\sigma}(\T^N) \right)$, so the same holds for the left-hand side. Hence we have uniqueness for the partial Fourier series described in Theorem \ref{Chinatown} and Remark \ref{The Pianist} on both sides. This fact, associated to the initial condition in \eqref{Schindler's List}, implies the following identities:
	\small{
	\begin{equation} \label{Los Angeles}
		\left\{ \begin{array} {rl}
			D_{t}\widehat{u}(t, \xi) + \left[a_{2}(t) |\xi|^{2} + \ds \sum_{j = 1}^{N} a_{1, j}(t) \xi_{j} +  a_{0}(t) \right] \widehat{u}(t, \xi) &= \widehat{f}(t, \xi), \\
			\widehat{u}(0, \xi) & = \widehat{g}(\xi),
		\end{array} \right. \ \ \forall t \in [0, T], \ \forall \xi \in \Z^{N}. 
	\end{equation}
	}
	\normalsize

	The argument above shows that for any \emph{barely regular} candidate $u$ to be a solution for $\eqref{Schindler's List}$, it is \emph{necessary} to solve the family of ordinary differential equations parametrized by elements in $\Z^N$ given in \eqref{Los Angeles}. Each equation in \eqref{Los Angeles} can be solved by the method of integrating factors. In fact, if we set for any index $k$
	\begin{equation} \label{Indiana Jones and the Last Crusade}
		a_{k}(t):= b_{k}(t) + ic_{k}(t),
	\end{equation}
	with both $b_{k}$ and $c_{k}$ real-valued, and denote their respective primitives by
	\begin{equation} \label{Goodfellas}
		A_{k}(t) := B_{k}(t) + iC_{k}(t), \ \ \text{where} \ \  B_{k}(t) := \ds \int_{0}^{t} b_{k}(s) ds \ \text{and} \ C_{k}(t) := \ds \int_{0}^{t} c_{kn}(s) ds,
	\end{equation}
	it is not difficult to prove that for any fixed $\xi \in \Z^N$
	\small{
		\begin{equation} \label{Sunset Boulevard} 
			\begin{split}
				\widehat{u}(t, \xi) &= \widehat{g}(\xi) \exp \left\{-i \left[A_{2}(t) |\xi|^{2} + \ds \sum_{j = 1}^{N} A_{1, j}(t) \xi_{j} +  A_{0}(t) \right]  \right\} + \\
				&+ \ds i \int_{0}^{t} \widehat{f}(s, \xi)  \exp \left\{i \left[\left(A_{2}(s) - A_{2}(t) \right) |\xi|^{2} + \ds \sum_{j = 1}^{N} \left(A_{1, j}(s) - A_{1, j}(t) \right) \xi_{j} +\left(A_{0}(s) - A_{0}(t) \right) \right]  \right\} ds, 
			\end{split}
		\end{equation}
	}
	\normalsize
	for every $t \in [0, T]$, and such solution is \emph{unique}.  
	
	\begin{Obs} \label{Requiem for a Dream}
		Even though this process has not yielded  any solution for \eqref{Schindler's List} yet, it has already proved that there is \emph{at most} one solution for the initial-value problem. 
	\end{Obs}

	\section{Normal Form} \label{The Wrestler}
	
	Our main objective in the present section is to show that, in order to study well-posedness for the problem \eqref{Schindler's List} with $P$ given in \ref{Her}, it is equivalent to work with a much simpler operator.  Before proceeding to the first statement, we introduce the concepts of well-posedness we intend to work with from now on. 
	
		\begin{Def} 
		Let $P(t,D_{t}, D_{x})$ as in \eqref{Her} and consider the initial-value problem \eqref{Schindler's List}.
		\begin{enumerate}[leftmargin=*]
			\item Given $\delta \geq 0$ we say that \eqref{Schindler's List} is  \emph{well-posed in $H^{r}$ with loss of  $\delta$ derivatives} if for every $r \in \R$, $f \in C \left([0,T]; H^{r}(\T^N) \right)$ and $g \in H^{r}(\T^N)$, there exists a \emph{unique} $u \in C \left([0,T]; H^{r-\delta}(\T^N) \right)$ which solves \eqref{Schindler's List}. In the case where we may take $\delta = 0$, we simply say that \eqref{Schindler's List} is \emph{well-posed} in $H^{r}$.
			
			\item We say that \eqref{Schindler's List} is \emph{well-posed in $C^{\infty}$} when for any $f \in C \left([0,T]; C^{\infty}(\T^N) \right)$ and $g \in C^{\infty}(\T^N)$, there exists a \emph{unique} $u \in C \left([0,T]; C^{\infty}(\T^N) \right)$ that solves \eqref{Schindler's List}. 
			
			\item Fixed $s \geq 1$, we say that \eqref{Schindler's List} is \emph{well-posed in $G^{s}$} if for any
			$f \in C \left([0,T]; G^{s}(\T^N) \right), \ g \in G^{s}(\T^N)$, there exists a \emph{unique} $u \in C \left([0,T]; G^{s}(\T^N) \right)$ solving \eqref{Schindler's List}. 
		\end{enumerate}
	\end{Def}
	
	\begin{Obs} \label{The Descendants}
	Suppose for a moment that one has proved well-posedness for any of the cases above. By going back to the description of the operator $P$, the following facts are immediate consequences: 
	
	\begin{itemize} [leftmargin=*]
		\item If $f \in C \left([0,T]; H^{r}(\T^N) \right)$ and $u \in C \left([0,T]; H^{r-\delta}(\T^N) \right)$  solves \eqref{Schindler's List},  $u \in C^{1} \left([0,T]; H^{r-\delta -2}(\T^N) \right)$. 
		\item If $f \in C \left([0,T]; C^{\infty}(\T^N) \right)$ and $u \in C \left([0,T]; C^{\infty}(\T^N) \right)$  solves \eqref{Schindler's List},  $u \in C^{1} \left([0,T]; C^{\infty}(\T^N) \right)$.
		\item If $f \in C \left([0,T]; G^{s}(\T^N) \right)$ and $u \in C \left([0,T]; G^{s}(\T^N) \right)$  solves \eqref{Schindler's List}, $u \in C^{1} \left([0,T]; G^{s}(\T^N) \right)$.
	\end{itemize}
	\end{Obs}

	\begin{Pro} \label{Seven}
		Let  $J: [0, T] \times \Z^{N} \to \C$ be the following continuous map: 
		\begin{equation} \label{The Aviator}
			J(t, \xi) = -B_{2}(t) |\xi|^{2} - \ds \sum_{j = 1}^{N} B_{1, j}(t) \xi_{i}  - A_{0}(t),
		\end{equation}
		with $B_{2}, B_{1, 1}, \ldots, B_{1, N}$ and $A_{0}$ as in \eqref{Goodfellas}.  Consider, for any $r \in \R$, 
		\begin{equation} \label{One Million Dollar Baby}  
			\begin{split}
				\Psi: C \left([0,T]; H^{r}(\T^N) \right) &\to C \left([0,T]; H^{r-2}(\T^N) \right) \\
				u = \ds \sum_{\xi \in \Z^N} \widehat{u}(t, \xi) e^{i x \xi} &\mapsto \Psi u = \ds \ds \sum_{\xi \in \Z^N} \left(e^{i J(t,\xi)}\widehat{u}(t, \xi) \right) e^{i x \xi}.
			\end{split} 
		\end{equation}
		Then $\Psi$ is well defined and the same holds for the restrictions
		\begin{equation} \label{Sicario}
			\begin{split}
				&\Psi: C \left([0,T]; C^{\infty}(\T^{N}) \right) \to C \left([0,T]; C^{\infty}(\T^N) \right); \\
				&\Psi: C \left([0,T]; G^{s}(\T^N) \right) \to C \left([0,T]; G^{s}(\T^N) \right), \ \forall s \geq 1.
			\end{split}
		\end{equation}
		Furthermore the maps in \eqref{Sicario} are bijective.
	\end{Pro}
	\begin{proof}
	We  prove first that \eqref{One Million Dollar Baby} is well defined.  Since $B_{2}$ and each $B_{1, j}$ is real, whilst $A_{0}$ is uniformly bounded, there exists a constant $C > 0$ such that 
		\begin{equation} \label{Incendies} 
			\left|e^{i J(t,\xi)}\widehat{u}(t, \xi) \right| \leq C \left|\widehat{u}(t, \xi) \right|, \ \ \forall t \in [0,T], \ \forall \xi \in \Z.
		\end{equation}
		Thus $\Psi u (t, x) \in H^{r}(\T^N)$ for any fixed $t \in [0, T]$ and $\widehat{\Psi u}(t, \xi) = e^{i J(t,\xi)}\widehat{u}(t, \xi)$, for every $ \xi \in \Z^N. $ It remains to prove the continuity in $H^{r-2}(\T^N)$; given $t, t_{0} \in [0, T]$, 
		\small{
			\begin{equation}  \label{Arrival}
			\begin{split}
				\left\|\Psi u(t) - \Psi u(t_{0})   \right\|_{H^{r-2}(\T^N)}^{2} &= \ds \sum_{\xi \in \Z^N} \left|e^{i J(t,\xi)}\widehat{u}(t, \xi) - e^{i J(t_{0},\xi)}\widehat{u}(t_{0}, \xi) \right|^{2} (1+|\xi|)^{2(r-2)} \\
				&\leq 2C^{2} \left\|u(t) - u(t_{0})   \right\|_{H^{r-2}(\T^N)}^{2} + 2\ds \sum_{\xi \in \Z^N} \left|e^{i J(t,\xi)} -  e^{i J(t_{0},\xi)} \right|^{2} |\widehat{u}(t_{0}, \xi)|^{2} (1+|\xi|)^{2(r-2)}. 
			\end{split}
			\end{equation}
		}
		\normalsize
		Next we apply mean value inequality to estimate the right-hand side of the expression above:  
		\small{
		\begin{align}
			\left|e^{i J(t, \xi)} - e^{i J(t_0, \xi)} \right| &\leq |t - t_{0}|  \sup_{\tau \in [t_{0}, t] } \left|- b_{2}(\tau) |\xi|^{2} -\ds \sum_{j = 1}^{N} b_{1, j}(\tau) \xi_{j} - a_{0}(\tau) \right| \sup_{\tau \in [t_{0}, t] } \left|e^{i J(\tau, \xi)} \right| \nonumber  \\
			&\leq C_{2} (1+ |\xi|)^{2} |t - t_{0}|, \label{Prisoners} 
		\end{align}
		}
		\normalsize
		for some  constant $C_{2}$ which can be taken  \emph{independently} of $t, t_{0}$ e $\xi$. By associating \eqref{Arrival} to \eqref{Prisoners}, we infer that 
		\begin{align*}
			\left\|\Psi u(t) - \Psi u(t_{0})   \right\|_{H^{r-2}(\T^N)}^{2} &\leq C_{3} \left(\left\|u(t) - u(t_{0})   \right\|_{H^{r}(\T^N)}^{2} + |t-t_{0}|^{2} \max_{\tau \in [0, T]} \left\|u (\tau) \right\|_{H^{r}(\T^N)}^{2} \right),
		\end{align*}
		for some $C_{3}$ that does not depend on $t$ nor $t_{0}$, which proves our first statement. The fact that $\Psi$ take elements of $C \left([0,T]; C^{\infty}(\T^N) \right)$ into itself is a direct consequence from the statement we have just proved and Theorem \ref{The Godfather II}.
		
		We proceed to the \emph{Gevrey} case; if $u \in C([0, T]; G^{s}(\T^{N}))$ by Theorem \ref{On the Waterfront} there exists  $h > 0$ such that $u: [0,T] \to G^{s}_{h}(\T^N) \ \text{is continuous}.$ Hence we find (by Lemma \ref{The People vs. Larry Flynt})  $\lambda (N, s) > 0$ such that
		\begin{align}
			|\widehat{u}(t, \xi)|  &\leq  2^{s} \left\|u(t) \right\|_{h, s} e^{- \frac{\lambda}{h^{1/s}} |\xi|^{1/s}} \leq  2^{s} \ds \max_{\tau \in [0, T]} \left(\left\|u(\tau) \right\|_{h, s} \right) e^{- \frac{\lambda}{h^{1/s}} |\xi|^{1/s}}, \ \ \ \forall \xi \in \Z^{N}, \ \forall t \in [0, T]. \label{Do the Right Thing} 
		\end{align}
		It follows from \eqref{Incendies} and \eqref{Do the Right Thing} that for each fixed $t \in [0, T]$ we have $\Psi u (t) \in G^{s}(\T^N)$ and
		\begin{equation*}
			\left| \widehat{\Psi u}(t, \xi) \right| \leq  C 2^{s}  \ds \max_{\tau \in [0, T]} \left(\left\|u(\tau) \right\|_{h, s} \right) e^{- \frac{\lambda}{h^{1/s}} |\xi|^{1/s}}, \ \ \ \forall \xi \in \Z^{N}, \ \forall t \in [0, T].
		\end{equation*}
		We deduce from Lemma \ref{Boyhood} that 
		\begin{equation*}
			\Psi u(t) \in G^{s}_{\kappa}(\T^N), \ \ \ \forall t \in [0, T], \ \ \text{where $\kappa = (2s/\lambda)^{s} h > 0$ \emph{does not depend on} $t$}. 
		\end{equation*}
		
		It remains to prove continuity: applying \eqref{Incendies}, \eqref{Prisoners} and \eqref{Do the Right Thing} we obtain
		\small{
		\begin{align}
			\left|\widehat{\Psi u}(t, \xi) - \widehat{\Psi u}(t_{0}, \xi)  \right| 
			&\leq  C \left|\widehat{u}(t, \xi) - \widehat{u}(t_{0}, \xi)  \right| +  C_{2}  2^{s} \ds \max_{\tau \in [0, T]} \left(\left\|u(\tau) \right\|_{h, s} \right) |t - t_{0}| (1+ |\xi|)^{2}  e^{- \frac{\lambda}{h^{1/s}} |\xi|^{1/s}} \label{There Will be Blood}.
		\end{align}
		}
		\normalsize
		On the other hand, it follows from Lemma \ref{The People vs. Larry Flynt} that
		\begin{equation} \label{Magnolia} 
			|\widehat{u}(t, \xi) - \widehat{u}(t_{0}, \xi)|  \leq  2^{s} \left\|u(t) - u(t_0) \right\|_{h, s} e^{- \frac{\lambda}{h^{1/s}} |\xi|^{1/s}}, \ \ \forall \xi \in \Z^N. 
		\end{equation}
		Finally, recall that there exists $C_{4}> 0$ such that 
		\begin{equation} \label{Boogie Nights} 
			(1+|\xi|^{2}) \leq C_{4} e^{ \frac{\lambda}{2h^{1/s}} |\xi|^{1/s}}, \ \ \forall  \xi \in \Z^N. 
		\end{equation}  
		Putting \eqref{Magnolia} and \eqref{Boogie Nights} into \eqref{There Will be Blood}, we deduce the existence of $M > 0$ such that
		\begin{equation*}
			\left|\widehat{\Psi u}(t, \xi) - \widehat{\Psi u}(t_{0}, \xi)  \right| \leq M \left(\left\|u(t) - u(t_0) \right\|_{h, s}  + |t - t_{0}| \right) e^{- \frac{\lambda}{2 h^{1/s}} |\xi|^{1/s}}, \ \ \forall  \xi \in \Z^N.
		\end{equation*}
		Using Lemma \ref{Boyhood}, one deduces that $\Psi u(t) \to  \Psi u(t_0)$ in $G^{s}(\T^N)$ when $t \to t_{0}$, as we intended to prove. 
		
		In order to check  that $\Psi$ is bijective in both smooth and Gevrey cases, define 
		\begin{equation} \label{Good Night, and Good Luck}   
			\begin{split}
				\Phi: C \left([0,T]; H^{r}(\T^N) \right) &\to C \left([0,T]; H^{r-2}(\T^N) \right) \\
				u = \ds \sum_{\xi \in \Z^N} \widehat{u}(t, \xi) e^{i x \xi} &\mapsto \Psi u = \ds \ds \sum_{\xi \in \Z^N} \left(e^{-i J(t,\xi)}\widehat{u}(t, \xi) \right) e^{i x \xi}.
			\end{split} 
		\end{equation}
		Then properties stated and proved for $\Psi$ also hold for $\Phi$. Moreover it follows immediately that $\Phi \circ \Psi = \Psi \circ \Phi = I$, which finalizes the proof. 
	\end{proof}

	\begin{Teo} \label{Cries and Whispers} 
		Let $P(t, D_{x}, D_{t})$ be the operator described  in \eqref{Her} and set
		\begin{equation} \label{Raiders of the Lost Ark} 
			\widetilde{P}(t, x, D_{x}, D_{t}) = D_{t} - ic_{2}(t) \Delta_{x}^{2} + \ds \sum_{j = 1}^{N} i c_{1, j}(t) D_{x_{j}},
		\end{equation}
		with $c_{2}$ and $c_{1, 1}, \ldots, c_{1, N}$ defined in \eqref{Indiana Jones and the Last Crusade}. Then \eqref{Schindler's List} is well-posed in $ C^{\infty}$ or $G^{s}$, for any $s \geq 1$, if and only if the same holds true for
		\begin{equation} \label{Wild Strawberries} 
			\left\{ \begin{array} {rl}
				\widetilde{P}u(t,x) &= \alpha(t,x), \\
				u(0, x) & = \beta(x), 
			\end{array} \right. \ \ \forall t \in [0, T], \ \forall x \in \T^N.
		\end{equation}
		Moreover, if there exists $\delta \geq 0$ such that \eqref{Schindler's List} is well-posed in $H^{r}$ with loss of $\delta$ derivatives, then \eqref{Wild Strawberries} is well-posed in $H^{r}$ with loss of  $(4+ \delta)$ derivatives. 
	\end{Teo}
	
	\begin{proof}
	Since all the results are proved similarly (using that $\Psi$ defined in \eqref{Sicario} is invertible), we restrict ourselves to the Sobolev case. 	Suppose \eqref{Schindler's List} is well-posed in $H^{r}$ with loss of $\delta$ derivatives and fix $\alpha(t,x) \in C \left([0, T]; H^{r}(\T^N) \right)$,  $\beta(x) \in H^{r}(\T^N)$. Let $\Psi$ as in Proposition \ref{Seven}; then $\Psi (\alpha) \in  C \left([0, T]; H^{r-2}(\T^N) \right)$  and there exists (by hypothesis) $v \in C \left([0, T]; H^{r- 2 - \delta}(\T^N) \right)$ such that 
		\begin{equation*} 
			\left\{\begin{array} {rl}
				Pv(t,x) &= \Psi (\alpha)(t,x), \\
				v(0, x) & = \beta(x), 
			\end{array} \right. \ \ \forall t \in [0, T], \ \forall x \in \T^{N}.
		\end{equation*}
		Thus $\alpha = \Psi^{-1} (Pv) \in C \left([0, T]; H^{r}(\T^N) \right)$ and we can write
		\small{
		\begin{equation*}
			\alpha = \Psi^{-1} (Pv) = \ds \sum_{\xi \in \Z^{N}} e^{- i J(t,\xi)} \left[ D_{t}\widehat{v}(t, \xi) + \left(a_{2}(t) |\xi|^{2} + \ds \sum_{j = 1}^{N} a_{1, j}(t) \xi_{j} +  a_{0}(t) \right) \widehat{v}(t, \xi) \right] e^{ix\xi}.
		\end{equation*}  
		}
		\normalsize
		On the other hand, $\Psi^{-1} v \in  C \left([0, T]; H^{r- 4 - \delta}(\T^N) \right)$.  By Remarks \ref{The Pianist} and \ref{Munich}, we are able to write
		\small{
			\begin{align*}
				\widetilde{P} (\Psi^{-1} v) &=\ds \sum_{\xi \in \Z^{N}} \widetilde{P} \left[ \left(e^{- i J(t,\xi)}\widehat{v}(t, \xi) \right) e^{i x \xi} \right] \\
				&= \ds \sum_{\xi \in \Z^{N}} e^{- i J(t,\xi)} \left[D_{t}\widehat{v}(t, \xi) + \left(- \ds \frac{\partial J}{\partial t}(t,\xi) + ic_{2}(t) |\xi|^{2} + \ds \sum_{j = 1}^{N} i c_{1, j}(t) \xi_{j} \right) \widehat{v}(t, \xi)  \right] e^{i x \xi} \\
				&= \ds \sum_{\xi \in \Z^{N}} e^{- i J(t,\xi)} \left[D_{t}\widehat{v}(t, \xi) + \left( (b_{2}(t) + ic_{2}(t)) |\xi|^{2} + \ds \sum_{j = 1}^{N} (b_{1, j}(t) + i c_{1, j}(t)) \xi_{i}  + a_{0}(t)  \right) \widehat{v}(t, \xi)  \right] e^{i x \xi}\\
				&=  \ds \sum_{\xi \in \Z^{N}} e^{- i J(t,\xi)} \left[ D_{t}\widehat{v}(t, \xi) + \left(a_{2}(t) |\xi|^{2} + \ds \sum_{j = 1}^{N} a_{1, j}(t) \xi_{j} +  a_{0}(t) \right) \widehat{v}(t, \xi) \right] e^{ix\xi} = \alpha.  
		\end{align*}}
		\normalsize 
		
		Therefore, by setting $u := \Psi^{-1} v \in C \left([0, T]; H^{r- 4 - \delta}(\T^N) \right)$, we deduce that $\widetilde{P} u = \alpha$. Furthermore it follows from expressions of $J$ and $\Psi$ (see \eqref{The Aviator} and \eqref{Good Night, and Good Luck} respectively) that $J(0, \xi) = 0$ for every $\xi \in \Z^{N}$, which implies that $u(0,x) = v(0,x) = \beta(x)$ for any $x \in \T^N$, and $u$ is a solution for \eqref{Wild Strawberries}. The well-posedness of the same problem is a consequence of Remark \ref{Requiem for a Dream}.

	\end{proof}
	
\section{A characterization of well-posedness through Fourier coefficients and its consequences} \label{Walk the Line} 

We begin the section by showing that, in order to obtain well-posedness for \eqref{Schindler's List}, it is sufficient to exhibit a priori inequalities which are reminiscent of \emph{energy estimates} for the Fourier coefficients \eqref{Sunset Boulevard} of the (only) candidate $u$. Next  we combine the result with Theorem \ref{Cries and Whispers} to obtain relations between well-posedness in different frameworks. Finally, we show that well-posedness in Sobolev spaces  can be reduced to well-posedness in $L^{2}$.

\begin{Lem} \label{American Gangster}  
	For arbitrary $t \in [0, T]$ and $\xi \in \Z^N$, consider $\widehat{u}(t, \xi)$ given by formula \eqref{Sunset Boulevard}. 
	\begin{enumerate}[label=\Roman*),  wide,  labelindent=0pt]
		\item Let $r \in \R$,  $f \in C\left([0,T]; H^{r}(\T^N)\right)$ and $g \in H^{r}(\T^N)$. If there exist constants  $C > 0$, $\rho \geq 0$ such that 
		\small{
		\begin{equation} \label{Zero Dark Thirty} 
			|\widehat{u}(t, \xi)| \leq C (1+|\xi|)^{\rho} \left[|\widehat{g}(\xi)|  + \left( \ds  \int_{0}^{T} |\widehat{f}(s, \xi)|^{2} ds \right)^{1/2} \right], \ \forall t \in [0,T], \ \forall \xi \in \Z^{N},
		\end{equation}
		}
		\normalsize
		then for each fixed $t \in [0, T]$ we have $u = \ds \sum_{\xi \in \Z^N} \widehat{u}(t, \xi) e^{i x \xi}$ in  $H^{r - \rho}(\T^N)$ and $u \in C\left([0, T]; H^{r - 2 - \rho}(\T^N) \right)$. 
		\item  Suppose $f \in C \left([0, T]; G^{s}(\T^N) \right)$ and $g \in G^{s}(\T^N)$, for some $s \geq 1$. In case there exist $C, \delta > 0$ such that 
		\begin{equation} \label{The Hurt Locker}
			|\widehat{u}(t, \xi)| \leq C e^{- \delta |\xi|^{1/s}}, \ \forall t \in [0, T], \ \forall \xi \in \Z^{N},	
		\end{equation}	
		it follows that $u = \ds \sum_{\xi \in \Z^N} \widehat{u}(t, \xi) e^{i x \xi}$ belongs to $C\left([0, T]; G^{s}(\T^N) \right)$.
	\end{enumerate}
\end{Lem}  

\begin{proof}
	If $f \in C\left([0,T]; H^{r}(\T^N)\right)$ and $g \in H^{r}(\T^N)$, there exists $A > 0$ such that
	\small{
	\begin{equation} \label{Children of Men} 
			\ds \sum_{\xi \in \Z^N} |\widehat{g}(\xi)|^{2} \left(1+ |\xi|\right)^{2r} \leq A \ \ \text{and} \ \  
			\max_{s \in [0, T]} \left(\ds \sum_{\xi \in \Z^N}  |\widehat{f}(s, \xi)|^{2} \left(1+ |\xi|\right)^{2r} \right) \leq A.
	\end{equation}
	}
	\normalsize
	Applying \eqref{Zero Dark Thirty} and Tonelli's theorem, we infer that 
	\small{
	\begin{align*}
		\ds \sum_{\xi \in \Z^N} |\widehat{u}(t, \xi)|^{2} \left(1+ |\xi|\right)^{2r - 2 \rho} &\leq 2C^{2} \ds \sum_{\xi \in \Z^N}  |\widehat{g}(\xi)|^{2} \left(1+ |\xi|\right)^{2r} + 2C^{2} \ds  \int_{0}^{T}  \left(\ds \sum_{\xi \in \Z^N}  |\widehat{f}(s, \xi)|^{2} \left(1+ |\xi|\right)^{2r} \right) ds \\
		&\leq 2AC^{2} (1 + T),  \ \ \forall t \in [0, T],
	\end{align*}
}
\normalsize
	which allows us to deduce that, for each fixed $t \in [0, T]$, $u(t, \cdot ) \in  H^{r - \rho} (\T^N)$.
	
	In order to prove continuity, take arbitrary $t_{0}, t_{1} \in [0,T]$. As a consequence of \eqref{Los Angeles} and the Fundamental Theorem of Calculus, we obtain for any  fixed $\xi \in \Z^N$
	\footnotesize{
		\begin{align}
			\widehat{u}(t_{1}, \xi) - \widehat{u}(t_{0}, \xi) &= (t_{1}-t_{0}) \ds \int_{0}^{1} \Bigg[ \Bigg. -i \left(a_{2}\left(t_{0} + y (t_{1} - t_{0}) \right) |\xi|^{2} + \ds \sum_{j = 1}^{N} a_{1, j}\left(t_{0} + y (t_{1} - t_{0})\right) \xi_{j} +  a_{0}\left(t_{0} + y (t_{1} - t_{0}) \right) \right) \nonumber \\
			&\times  \widehat{u}\left(t_{0} + y (t_{1} - t_{0}), \xi \right) + i\widehat{f}\left(t_{0} + y (t_{1} - t_{0}), \xi \right) \Bigg. \Bigg] dy. \label{Saving Private Ryan} 
		\end{align}
	}
	\normalsize
	Since all functions	$a_{0}, a_{1, j}$ e $a_{2}$ are uniformly bounded we are able to find $B > 0$  \emph{independent} of $\xi$, $t_{0}$, $t_{1}$ and $y$ such that 
	\small{
		\begin{equation} \label{Minority Report} 
			\left| -i \left(a_{2}\left(t_{0} + y (t_{1} - t_{0}) \right) \xi^{2} + \ds \sum_{j = 1}^{N} a_{1, j}\left(t_{0} + y (t_{1} - t_{0})\right) \xi_{j} +  a_{0}\left(t_{0} + y (t_{1} - t_{0}) \right) \right) \right| \leq B (1+ |\xi|)^{2}, \ \ \forall \xi \in \Z^N. 
		\end{equation} 
	}
	\normalsize
	By associating \eqref{Saving Private Ryan} to \eqref{Minority Report}, besides using the hypothesis  and Hölder's inequality, we conclude that 
	\small{
	\begin{align*}
		|\widehat{u}(t_{1}, \xi) - \widehat{u}(t_{0}, \xi)| &\leq |t_{1}-t_{0}| \left[ \ds \int_{0}^{1} B (1+ |\xi|)^{2}  \left|\widehat{u}\left(t_{0} + y (t_{1} - t_{0}), \xi \right) \right| + \left|\widehat{f}\left(t_{0} + y (t_{1} - t_{0}), \xi \right) \right| dy \right] \\
		&\leq |t_{1}-t_{0}| \ds \int_{0}^{1} B (1+ |\xi|)^{2} C (1+|\xi|)^{\rho} \left[|\widehat{g}(\xi)|  + \left( \ds  \int_{0}^{T} |\widehat{f}(s, \xi)|^{2} ds \right)^{1/2} \right] dy \ + \\
		&+ |t_{1}-t_{0}| \ds \int_{0}^{1} |\widehat{f}\left(t_{0} + y (t_{1} - t_{0}), \xi \right)| dy \\
		&\leq |t_{1}-t_{0}| \left[ BC (1+ |\xi|)^{2 + \rho} \left[|\widehat{g}(\xi)|  + \left( \ds  \int_{0}^{T} |\widehat{f}(s, \xi)|^{2} ds \right)^{1/2} \right] + \left( \ds  \int_{0}^{T} |\widehat{f}(s, \xi)| ds \right) \right] \\
		&\leq |t_{1}-t_{0}| D (1+ |\xi|)^{2 + \rho} \left[|\widehat{g}(\xi)|  + \left( \ds  \int_{0}^{T} |\widehat{f}(s, \xi)|^{2} ds \right)^{1/2} \right], \ \ \ \ \ \forall \xi \in \Z^N, 
	\end{align*}
	}
	\normalsize	
for some $D > 0$ which only depends on $B$, $C$ and $T$. 
	
	It follows from \eqref{Children of Men} and Tonelli's Theorem that
	\begin{align*}
		\ds \sum_{\xi \in \Z^N}|\widehat{u}(t_{1}, \xi) - \widehat{u}(t_{0}, \xi)|^{2} (1+ |\xi|)^{2r - 4 - 2 \rho} &\leq \ds \sum_{\xi \in \Z^N} |t_{1}-t_{0}|^{2} D^{2} \left[|\widehat{g}(\xi)|^{2}  + \left( \ds  \int_{0}^{T} |\widehat{f}(s, \xi)|^{2} ds \right) \right] (1+ |\xi|)^{2r}  \\
		&\leq |t_{1}-t_{0}|^{2} AD^{2} + |t_{1}-t_{0}|^{2} ATD^{2}.
	\end{align*}
	Thus, for any $t_{0}, t_{1} \in [0, T]$, we obtain $\left\|u(t_{1}, x) - u(t_{0}, x) \right\|_{r -2 - \rho} \leq M |t_{1} - t_{0}|$, for a constant $M > 0$ that only depends on $T$, $A$, $B$ and $C$. Therefore  $u \in C\left([0, T]; H^{r - 2 - \rho}(\T^N) \right)$, as we intended to prove.  
	
	We advance to the Gevrey setting; it follows immediately from \eqref{The Hurt Locker} that, for every fixed  $t \in [0, T]$,  $u(t, x) \in G^{s}(\T^N)$. It remains to prove continuity; by proceeding analogously to the first case, one has
	\begin{equation} \label{Taxi Driver}
		|\widehat{u}(t_{1}, \xi) - \widehat{u}(t_{0}, \xi)| \leq |t_{1}-t_{0}| \left[ \ds \int_{0}^{1} B (1+ |\xi|)^{2}  |\widehat{u}\left(t_{0} + y (t_{1} - t_{0}), \xi \right)| + |\widehat{f}\left(t_{0} + y (t_{1} - t_{0}), \xi \right)| \ dy \right]. 
	\end{equation}
Applying once again Lemma \ref{The People vs. Larry Flynt} and Theorem \ref{On the Waterfront}, there exist $C', \sigma > 0$ such that
	\begin{equation} \label{The Color Purple}
		|\widehat{f}(t, \xi)| \leq C' e^{- \sigma |\xi|^{1/s}}, \ \forall t \in [0, T], \ \forall \xi \in \Z^N. 
	\end{equation}
	Take $C_{1} = \max \left\{C, C' \right\}$ and $\omega = \min \left\{\sigma, \delta \right\}$; it is a consequence of \eqref{The Hurt Locker}, \eqref{Taxi Driver} and \eqref{The Color Purple} that
	\begin{align*}
		|\widehat{u}(t_{1}, \xi) - \widehat{u}(t_{0}, \xi)| &\leq |t_{1}-t_{0}| C_{2} e^{- \omega/2 |\xi|^{1/s}}, \ \ \ \ \forall \xi \in \Z^N,
	\end{align*}
	where the constant $C_{2}$ depends only on $B, C, C', \sigma$ and $\delta$. The result is a direct consequence of Lemma \ref{Boyhood} and the proof is complete. 
\end{proof}

\begin{Pro} \label{The Graduate} 
	Let $P$ be the operator described in \eqref{Her}. 
	\begin{enumerate}[label=\alph*),  wide,  labelindent=0pt ]
		\item \label{Back to the Future}  If there is  $\delta \geq  0$ such that  \eqref{Schindler's List} is well-posed in $H^{r}$ with loss of $\delta$ derivatives then it is well-posed in $C^{\infty}$. 
		\item \label{Back to the Future II} If \eqref{Schindler's List} is well-posed in $C^{\infty}$, it is also well-posed in $G^{s}$, for every  $s \geq 1$. 
		\item \label{Back to the Future III} Suppose  $s_{1}, s_{2} \in \R$ such that $1 \leq s_{2} \leq s_{1}$. If  \eqref{Schindler's List} is well-posed in $G^{s_{1}}$, the same holds in $G^{s_{2}}$.	
	\end{enumerate}
\end{Pro}

\begin{proof}
Implication  \ref{Back to the Future} is pretty standard; we proceed to statement \ref{Back to the Future II}. If \eqref{Schindler's List} is well-posed in $C^{\infty}$ it follows from  Theorem \ref{Cries and Whispers} the well-posedness in $C^{\infty}$ for \eqref{Wild Strawberries}, with $\widetilde{P}$ set in \eqref{Raiders of the Lost Ark}. Fix $s \geq 1$ and take 
	\begin{equation*}
		f_{1} \equiv 0, \ \  g_{1}(x) = \ds \sum_{\xi \in \Z^N} e^{-|\xi|^{1/2s}}  e^{i x \xi}.
	\end{equation*} 
	Then $g_{1} \in C^{\infty} (\T^N)$ and $\widehat{g_{1}}(\xi) = e^{-|\xi|^{1/2s}}$, for each $\xi \in \Z^N$. By hypothesis there exists a unique $u_{1} \in C \left([0,T]; C^{\infty}(\T^N) \right)$ which solves
	\begin{equation*}  
		\left\{ \begin{array} {rl}
			\widetilde{P}u_{1}(t,x) &= 0, \\
			u_{1}(0, x) & = g_{1}(x), 
		\end{array} \right. \ \ \forall t \in [0, T], \ \forall x \in \T^N.
	\end{equation*}
	In this particular case, we write $u_{1} = \ds \sum_{\xi \in \Z^{N}} \widehat{u_{1}}(t, \xi) e^{i x \xi}$  and obtain (analogously from \eqref{Sunset Boulevard}) 
	\small{
	\begin{equation} \label{Ben-Hur} 
		\widehat{u_{1}}(t, \xi) = e^{-|\xi|^{1/2s}} \underbrace{\exp \left( C_{2}(t) |\xi|^{2} + \ds \sum_{j = 1}^{N} C_{1, j}(t) \xi_{j}\right)}_{:= \varUpsilon_{1} (t, \xi)}, \ \ \ \forall t \in [0, T], \ \ \forall \xi \in \Z^N.
	\end{equation}
	}
	\normalsize
By the fact that $u_{1} \in C \left([0,T]; C^{\infty}(\T^N) \right)$, one infers the existence of $M_{1} > 0$ such that
	\begin{equation}  \label{The Bridge on the River Kwai} 
		\varUpsilon_{1} (t, \xi)  \leq M_{1} e^{|\xi|^{1/2s}} \ \ \ \forall t \in [0, T], \ \ \forall \xi \in \Z^N. 
	\end{equation}

	Now set	 $f_{2}(t,x) = - i \ds \sum_{\xi \in \Z^N}  e^{-|\xi|^{1/2s}} e^{i x \xi}$ and $g_{2} \equiv 0$. Similarly,  $f_{2} \in C \left([0, T]; C^{\infty}(\T^N) \right)$ and there exists a unique $u_{2} = \ds \sum_{\xi \in \Z^{N}} \widehat{u_{2}}(t, \xi) e^{i x \xi}$ in  $C \left([0, T]; C^{\infty}(\T^N) \right)$ for which
	\begin{equation*}  
		\left\{ \begin{array} {rl}
			\widetilde{P}u_{2}(t,x) &= f_{2}(t,x) \\
			u_{2}(0, x) & = 0
		\end{array} \right. \ \ \forall t \in [0, T], \ \forall x \in \T^N.
	\end{equation*}
	Then
	\small{
	\begin{equation} \label{Gandhi} 
		\widehat{u_{2}}(t, \xi) = \ds e^{-|\xi|^{1/2s}} \underbrace {\int_{0}^{t}   \exp  \left[\left(C_{2}(t) - C_{2}(s) \right) |\xi|^{2} + \ds \sum_{j = 1}^{N} \left(C_{1, j}(t) - C_{1, j}(s) \right) \xi_{j} \right]   ds}_{:= \varUpsilon_{2} (t, \xi)},  \ \ \ \forall t \in [0, T], \ \ \forall \xi \in \Z^N.
	\end{equation}
	}
	\normalsize
	and there exists $M_{2} > 0$ such that 
	\begin{equation} \label{All About Eve}
		\varUpsilon_{2} (t, \xi)  \leq M_{2} e^{|\xi|^{1/2s}} \ \ \ \forall t \in [0, T], \ \ \forall \xi \in \Z^N. 
	\end{equation}
	
	 Suppose $f \in C \left([0, T]; G^{s}(\T^N) \right)$ e $g \in G^{s}(\T^N)$. Then there exist $M, \delta > 0$ such that  
	\begin{equation} \label{Zodiac}
		\begin{split}
			|\widehat{g}(\xi)|\leq M e^{- \delta |\xi|^{1/s}} \ \ \text{and} \ \  |\widehat{f}(t, \xi)| \leq M e^{- \delta |\xi|^{1/s}}, \ \ \ \forall t \in [0, T], \ \forall \xi \in \Z^N.
		\end{split}
	\end{equation} 
	Let $u \in C \left([0, T]; C^{\infty}(\T^N) \right)$ be the unique solution for 
	\begin{equation*}  
		\left\{ \begin{array} {rl}
			\widetilde{P}u(t,x) &= f(t,x), \\
			u(0, x) & = g(x),
		\end{array} \right. \ \ \forall t \in [0, T], \ \forall x \in \T^N.
	\end{equation*}
	It follows from \eqref{Sunset Boulevard} that
	\small{
		\begin{equation} \label{Lawrence of Arabia} 
			\begin{split}
				\widehat{u}(t, \xi) &= \widehat{g}(\xi) \exp \left(C_{2}(t) |\xi|^{2} + \ds \sum_{j = 1}^{N} C_{1, j}(t) \xi_{j} \right)  + \\
				&+ \ds i \int_{0}^{t} \widehat{f}(s, \xi)  \exp  \left[\left(C_{2}(t) - C_{2}(s) \right) |\xi|^{2} + \ds \sum_{j = 1}^{N} \left(C_{1, j}(t) - C_{1, j}(s) \right) \xi_{j}  \right]   ds,  \ \ \ \forall t \in [0, T], \ \ \forall \xi \in \Z^N.
			\end{split}
		\end{equation}
	}
	\normalsize
	By associating \eqref{Lawrence of Arabia} to \eqref{Ben-Hur}, \eqref{The Bridge on the River Kwai}, \eqref{Gandhi}, \eqref{All About Eve} and \eqref{Zodiac}, we deduce that
	\begin{equation} \label{Jaws}
		\begin{split}
			|\widehat{u}(t, \xi)| & \leq \left|\widehat{g}(\xi) \right|  \varUpsilon_{1} (t, \xi)  +  \left(\ds \max_{r \in [0, T]} |\widehat{f}(r, \xi)| \right)  \varUpsilon_{2} (r, \xi)  \leq  M e^{- \delta |\xi|^{1/s}} \left( M_{1} e^{|\xi|^{1/2s}} +  M_{2} e^{|\xi|^{1/2s}}  \right)   \\
			&\leq M_{3} e^{- \delta |\xi|^{1/s}} e^{|\xi|^{1/2s}}  \leq M_{4} e^{- \frac{\delta}{2} |\xi|^{1/s}}, \ \ \forall t \in [0,T], \ \forall \xi \in \Z^N,
		\end{split}
	\end{equation}
	for some constant $M_{4} > 0$ that depends only on $M, M_{1}, M_{2}$ and $\delta$. We infer from Lemma \ref{American Gangster} that the initial-value problem for $\widetilde{P}$ is well-posed in $G^{s}$. The result is then a consequence of Theorem \ref{Cries and Whispers}. Finally, \ref{Back to the Future III} is proved in an analogous manner. 
\end{proof}

\begin{Cor} \label{Roma}
	If the problem \eqref{Schindler's List} is not well-posed in $C^{\omega} = G^{1}$, it is not well-posed in any of the $G^{s}$,  $C^{\infty}$ or $H^{r}$ spaces.
\end{Cor}

Since the coefficients operator $P$ do not depend on space variables, it is not difficult to verify the next result.

\begin{Pro} 
	Consider $P$ the operator given in \eqref{Her} and assume the existence of $\kappa \in \R$, $\sigma \geq 0$ such that for any $f \in C \left([0, T]; H^{\kappa}(\T^N)\right)$ and $g \in H^{\kappa}(\T^N)$ there exists  $u \in C\left([0, T]; H^{\kappa - \sigma}(\T^N) \right)$ which solves \eqref{Schindler's List}; then \eqref{Schindler's List} is well-posed in $H^{r}$ with loss of $\sigma$ derivatives. In particular, well-posedness in $H^{r}$ is equivalent to well-posedness in $L^{2}$. 
\end{Pro}

\section{Ill-posedness for a family of operators} \label{Memento}
	
	In the present section we show that if $\Im (a_{2})$ is positive at any point of the interval $[0, T]$, \eqref{Schindler's List} is \emph{never} well-posed in $C^{\omega}$ and hence in any of the settings previously defined (Corollary \ref{Roma}). Recall that Theorem \eqref{Cries and Whispers} allows us to consider a \emph{simpler} operator, described in \eqref{Raiders of the Lost Ark}. In this situation, if $u$ solves
	\begin{equation} \label{The Lord of the Rings: The Fellowship of the Ring}
		\left\{ \begin{array} {rl}
			\widetilde{P}u(t,x) &= f(t,x), \\
			u(0, x) & = g(x), 
		\end{array} \right. \ \ \forall t \in [0, T], \ \forall x \in \T^{N},
	\end{equation}
	it follows from a analogous computation made in  \eqref{Sunset Boulevard} that
	\small{
	\begin{equation} \label{Forrest Gump} 
		\begin{split}
			\widehat{u}(t, \xi) &= \widehat{g}(\xi) \exp  \left[C_{2}(t) |\xi|^{2} + \ds \sum_{j = 1}^{N} C_{1, j}(t) \xi_{j}  \right]  + \\
			&+ \ds i \int_{0}^{t} \widehat{f}(s, \xi)  \exp  \left[\left(C_{2}(t) - C_{2}(s) \right) |\xi|^{2} + \ds \sum_{j = 1}^{N} \left(C_{1, j}(t) - C_{1, j}(s) \right) \xi_{j}  \right]   ds,  \ \ \forall t \in [0, T], \ \ \forall \xi \in \Z^N.
		\end{split}
	\end{equation}
}
\normalsize
	\begin{Teo} \label{The Lord of the Rings: The Return of the King} Suppose  $t^{\star} \in [0,T]$ such that $c_{2}(t^{\star}) > 0$. Then there exist  $f(t,x) \in C \left([0, T]; C^{\omega}(\T^N) \right)$ and $g(x) \in C^{\omega} (\T^N)$ for which there is no solution in $C\left([0, T]; C^{\omega}(\T^N) \right)$ for \eqref{The Lord of the Rings: The Fellowship of the Ring}. 
	\end{Teo}
	
	\begin{proof}
	Since $c_{2}$ is continuous, we may assume $t^{\ast} > 0$. Furthermore we are able to find   $\delta, \varepsilon > 0$ such that $[t^{\star} - \delta, t^{\star}] \subset [0, T]$ and
		\begin{equation} \label{The Lord of the Rings: The Two Towers}
			c_{2}(t) \geq \varepsilon, \ \ \forall t \in [t^{\star} - \delta, t^{\star}].
		\end{equation}
		Let  $g(x) \equiv 0$ and $f(t, x) = f(x) = \ds \sum_{\xi \in \Z^N} e^{-|\xi|} e^{ix \xi}.$  It follows from \eqref{Forrest Gump} that
		\small{
		\begin{equation} \label{The Prestige}
			|\widehat{u}(t, \xi)| = e^{-|\xi|} \ds  \int_{0}^{t} \exp  \left[\left(C_{2}(t) - C_{2}(s) \right) |\xi|^{2} + \ds \sum_{j = 1}^{N} \left(C_{1, j}(t) - C_{1, j}(s) \right) \xi_{j}  \right] ds, \ \ \text{ $\forall t \in [0,T]$, $\forall \xi \in \Z^N$}.
		\end{equation}
		}
		\normalsize

		On the other hand, observe that
		\begin{equation} \label{The Dark Knight Rises}
			C_{2}(t) - C_{2}(s) =  \ds \int_{0}^{t} c_{2}(r) dr - \ds \int_{0}^{s} c_{2}(r) dr =  \ds \int_{s}^{t} c_{2}(r) dr.
		\end{equation}
		Hence by associating \eqref{The Lord of the Rings: The Two Towers} and \eqref{The Dark Knight Rises} to \eqref{The Prestige} we infer that, for every $\xi \in \Z^N$,
		\small{
		\begin{align*}
			|\widehat{u}(t^{\star}, \xi)| &=  e^{-|\xi|} \ds \int_{0}^{t^{\star}} \exp  \left[\left(C_{2}(t^{\star}) - C_{2}(s) \right) |\xi|^{2} + \ds \sum_{j = 1}^{N} \left(C_{1, j}(t^{\star}) - C_{1, j}(s) \right) \xi_{j}  \right] ds  \\
			&\geq e^{-|\xi|} \ds \int_{t^{\star} - \delta}^{t^{\star} - \delta/2} \exp  \left[\left(C_{2}(t^{\star}) - C_{2}(s) \right) |\xi|^{2} + \ds \sum_{j = 1}^{N} \left(C_{1, j}(t^{\star}) - C_{1, j}(s) \right) \xi_{j}  \right] ds ds  \\
			&\geq e^{-|\xi|} \ds \int_{t^{\star} - \delta}^{t^{\star} - \delta/2} \exp \left[\varepsilon (t^{\star} - s)  |\xi|^{2} +  \ds \sum_{j = 1}^{N} \left(C_{1, j}(t^{\star}) - C_{1, j}(s) \right) \xi_{j} \right]   ds  \\
			&\geq \exp \left({\ds \frac{\varepsilon \delta}{2} |\xi|^{2} - |\xi| } \right) \ds \int_{t^{\star} - \delta}^{t^{\star} - \delta/2} \exp  \left[ \ds \sum_{j = 1}^{N} \left(C_{1, j}(t^{\star}) - C_{1, j}(s) \right) \xi_{j} \right]   ds.
		\end{align*}
	}
	\normalsize

		Furthermore $\left| C_{1_{j}}(t) \right| \leq  \ds \int_{0}^{t} |c_{1, j}(r)|dr \leq T \ds \max_{r \in [0, T]} |c_{1, j}(r)|$. This implies that, for some $M > 0$, 
		\begin{equation*}
			  |C_{1,j}(t)| \leq M, \ \ \ \forall t \in [0,T], \ \ \forall j \in \left\{1, 2, \ldots, N \right\}.
		\end{equation*}
		Hence
		\begin{align*}
			|\widehat{u}(t^{\star}, \xi)| &\geq \exp \left({\ds \frac{\varepsilon \delta}{2} |\xi|^{2} - |\xi| } \right) \ds \int_{t^{\star} - \delta}^{t^{\star} - \delta/2} \exp  \left(-2M |\xi|  \right)  ds \geq \ds \frac{\delta}{2} \exp \left({\ds \frac{\varepsilon \delta}{2} |\xi|^{2} - (2M+1)|\xi|} \right), \ \ \forall \xi \in \Z^N. 
		\end{align*}
		Therefore, there exists $n_{0} \in \N$ large enough so that	
		\begin{equation*}
			|\widehat{u}(t^{\star}, \xi)| \geq \ds \frac{\delta}{2} \exp \left(\ds \frac{\varepsilon \delta}{4} |\xi|^{2} \right), \ \ \ \forall \xi \in \Z^N; \ |\xi| \geq n_{0}.
		\end{equation*}
		Thence the formal solution $u = \ds \sum_{\xi \in \Z^N} \widehat{u}(t, \xi) e^{ix\xi}$ is not even a function when $t = t^{\star}$, ending the proof. 
	\end{proof}
	
	As a consequence of Theorem \eqref{Cries and Whispers},  we have just proved the following
	
	\begin{Teo} \label{American  Beauty}
		Let $P$ be the operator defined in \eqref{Her}. If there exists  $t^{\star} \in [0, T]$ such that
		\begin{equation*}
			\Im (a_{2}(t^{\star})) = c_{2}(t^{\star}) > 0,
		\end{equation*} 
		then \eqref{Schindler's List} is ill-posed in $C^{\omega}$, which implies it is ill-posed in any of the $G^{s}$, $C^{\infty}$ or $H^{r}$ spaces.
	\end{Teo}

	\begin{Cor} \label{The Seventh Seal}  
		In order for \eqref{Schindler's List} to be well-posed in any of the spaces aforementioned it is necessary that $\Im (a_{2}(t)) = c_{2}(t) \leq 0, \  \forall t \in [0,t].$
	\end{Cor}
	
Right below we show that such condition is far from sufficient. 
	
	\begin{Teo} \label{La La Land}
		Suppose the existence of $t^{\star} \in [0,T]$ and $\delta > 0$ such that
		\begin{itemize}[wide,  labelindent=0pt ]
			\item $c_{2} \equiv 0$ in $[t^{\star}, t^{\star} + \delta]$; 
			\item $c_{1, j} \not\equiv 0$ in $[t^{\star}, t^{\star} + \delta]$, for some $j \in \left\{1, 2, \ldots N \right\}$. 
		\end{itemize}  
		Then there exist  $f(t,x) \in C \left([0, T]; C^{\omega}(\T^N) \right)$ and $g(x) \in C^{\omega} (\T^N)$ for which there is no solution in $C\left([0, T]; C^{\omega}(\T^N) \right)$ for \eqref{The Lord of the Rings: The Fellowship of the Ring}.
	\end{Teo}	
	
	\begin{proof}
		With no loss of generality we may assume that $j = 1$ and (by continuity) that $c_{1, 1}$ is either strictly positive or negative in $[t^{\star}, t^{\star} + \delta]$. Consider the case where $c_{1, 1}$ is positive (the other is completely analogous); then there exists $\varsigma > 0$ such that $c_{1, 1}(t) \geq \varsigma$,  for every $ t \in [t^{\star}, t^{\star} + \delta].$ Define then 
		\begin{equation*}
			f(t,x) = - i \ds \sum_{\eta \in \Z_{+}} e^{- \frac{\delta \varsigma \eta}{4}} e^{i x_{1} \eta},  \ \ \   g(x) \equiv 0.  
		\end{equation*}
		Applying \eqref{Forrest Gump}, one obtains
		\small{
		\begin{equation*}   
			\widehat{u}(t, \eta, 0) = e^{- \frac{\delta \varsigma \eta}{4}} \ds \int_{0}^{t} \exp  \left[\left(C_{2}(t) - C_{2}(s) \right) \eta^{2} + \left(C_{1, 1}(t) - C_{1, 1}(s) \right) \eta  \right]  ds, \ \ \forall t \in [0, T], \ \forall \eta \in \Z_{+},
		\end{equation*}
		}
		\normalsize
		where $(t, \eta, 0)$ represents $(t, \eta, \underbrace{0, \ldots, 0}_{\in \Z^{N-1}})$. Therefore 
		\small{
		\begin{align*}
			|\widehat{u}(t^{\star} + \delta, \eta, 0)| 
			&\geq e^{- \frac{\delta \varsigma \eta}{4} } \ds \int_{t^{\star}}^{t^{\star} + \delta} \exp \left[\left( \underbrace{\ds \int_{s}^{t^{\star} + \delta} c_{2}(r) dr}_{\equiv 0} \right) \eta^{2} + \left(\ds \int_{s}^{t^{\star} + \delta} c_{1, 1}(r) dr \right) \eta  \right]   ds \\
			&=  e^{- \frac{\delta \varsigma \eta}{4} } \ds \int_{t^{\star}}^{t^{\star} + \delta} \exp \left[ \left(\ds \int_{s}^{t^{\star} + \delta} c_{1, 1}(r) dr \right) \eta  \right]   ds \geq  e^{- \frac{\delta \varsigma \eta}{4} } \ds \int_{t^{\star}}^{t^{\star} + \delta} \exp \left[ \left(\ds \int_{s}^{t^{\star} + \delta} \varsigma  dr \right) \eta  \right]   ds  \\
			&\geq e^{- \frac{\delta \varsigma \eta}{4} } \ds \int_{t^{\star}}^{t^{\star} + \delta} \exp \left[\varsigma \left(t^{\star} + \delta - s  \right) \eta  \right]   ds \geq e^{- \frac{\delta \varsigma \eta}{4} } \ds \int_{t^{\star}}^{t^{\star} + \frac{\delta}{2}} \exp \left[\varsigma \left(t^{\star} + \delta - s  \right) \eta  \right]   ds \\ 
			&\geq e^{- \frac{\delta \varsigma \eta}{4} } \ds \int_{t^{\star}}^{t^{\star} + \frac{\delta}{2}} e^{\frac{\varsigma  \delta \eta}{2}}     ds \geq \ds \frac{\delta e^{\frac{\delta \varsigma \eta}{4} }}{2} , \ \ \  \   \forall \eta \in \Z_{+}.
		\end{align*}
		}
		\normalsize
		Since $\delta, \varsigma$ and $\eta$ are non-negative numbers,  the formal solution $u = \ds \sum_{\xi \in \Z^N} \widehat{u}(t, \xi) e^{ix\xi}$ is not even a function when $t = t^{\star} + \delta$, which  closes the proof. 
	\end{proof}
	
	\begin{Cor} 
		Let $P$ be the operator defined in \eqref{Her}. If there exist  $t^{\star} \in [0, T]$ and $\delta > 0$ such that
		\begin{equation*}
			\begin{gathered}
				\Im (a_{2}) = c_{2} \equiv 0 \ \text{in} \ [t^{\star}, t^{\star} + \delta]; \ \ \  \Im (a_{1, j}) = c_{1, j} \not\equiv 0 \ \text{in} \ [t^{\star}, t^{\star} + \delta], \ \text{for some $j \in \left\{1, 2, \ldots N \right\}$}.
			\end{gathered}
		\end{equation*}
		Then \eqref{Schindler's List} is ill-posed in $C^{\omega}$, which implies it is ill-posed in any of the $G^{s}$, $C^{\infty}$ or $H^{r}$ spaces. 
	\end{Cor}

	\section{Well-posedness for a family of operators} \label{Gladiator} 
	
	So far we have only found families of operators for which problem \eqref{Schindler's List} is ill-posed.  In this subsection we prove well-posedness in all  aforementioned settings for two different cases, with the most important of them being the situation where $\Im (a_{2})$ is strictly negative everywhere.
	
	\begin{Obs}
	Since we could replace $P$ by $iP$ with no loss of generality in \eqref{Schindler's List}, we would have an operator in the form $\del_{t} + Q(t, D_{x})$, and there exist $C, R > 0$ such that 
	\begin{equation*}
\Re(Q(t, \xi)) \geq C |\xi|^{2}, \ \ \  |\xi| \geq R. 
	\end{equation*}
	This means that $Q$ is a strongly elliptic symbol and hence that $iP$ is a parabolic operator.  
	\end{Obs} 
	
	Recall that our only candidate for solution has its Fourier coefficients given in \eqref{Sunset Boulevard}.

	\begin{Teo} \label{Rear Window} 
		Let $P$ be the operator defined in \eqref{Her} and assume that
		\begin{equation*}
			\Im (a_{2})(t) = c_{2}(t) < 0,  \  \text{for every  $t \in [0, T]$}.
		\end{equation*}
		Then \eqref{Schindler's List} is well-posed in $H^{r}$.  
	\end{Teo}
	
	\begin{proof}
		Applying \eqref{Goodfellas} and \eqref{Sunset Boulevard}, it follows immediately that 
		\small{
			\begin{align} 
				|\widehat{u}(t, \xi)| &\leq |\widehat{g}(\xi)| \exp  \left(C_{2}(t) |\xi|^{2} + \ds \sum_{j = 1}^{N} C_{1, j}(t) \xi_{j} +  C_{0}(t) \right)  \nonumber + \\
				&+ \ds \int_{0}^{t} |\widehat{f}(s, \xi)|   \exp  \left[\left(C_{2}(t) - C_{2}(s) \right) |\xi|^{2} + \ds \sum_{j = 1}^{N} \left(C_{1, j}(t) - C_{1, j}(s) \right) \xi_{j} +\left(C_{0}(t) - C_{0}(s) \right) \right]   ds. \label{The Untouchables}
			\end{align}
		}
		\normalsize
		Using the hypothesis, the continuity of the coefficients and the compactness of $[0, T]$, we deduce the existence of $\varepsilon, M > 0$ such that 
		\begin{equation*}
			c_{2}(t) \leq - \varepsilon,  \ \ |c_{1, j}(t)| \leq M, \ \ |c_{0}(t)| \leq M, \ \ \ \ \ \forall j \in \left\{1, 2, \ldots, N \right\}, \ \ \forall t \in [0, T].
		\end{equation*}
		Hence for every $0 \leq s \leq  t \leq T$, we have
		\begin{equation} \label{Scarface}
			\begin{split}
				C_{2}(t) - C_{2}(s) \leq - \varepsilon (t-s), \ \ C_{1, j}(t) - C_{1, j}(s) \leq M (t -s),\ \ C_{0}(t) - C_{0}(s) \leq  M (t -s).
			\end{split}
		\end{equation}
		
		Associating \eqref{The Untouchables} to \eqref{Scarface}, one obtains
		\small{
		\begin{equation} \label{Le Scaphandre et le Papillon}
			\begin{split}
				|\widehat{u}(t, \xi)| \leq e^{MT} \left[|\widehat{g}(\xi)| \exp \left[t \left(- \varepsilon |\xi|^{2} + MN |\xi| \right)  \right]  + \ds \int_{0}^{t} |\widehat{f}(s, \xi)| \exp \left[(t-s) \left(- \varepsilon |\xi|^{2} + MN |\xi| \right) \right]  ds \right].
			\end{split}
		\end{equation}
		}
		\normalsize
		Since  $\varepsilon$ is positive,  there exists $L > 0$ such that 
		\begin{equation} \label{Marriage Story} 
			\left(-\varepsilon |\xi|^{2} + MN |\xi| \right) \leq L, \ \ \forall \xi \in \Z^N.
		\end{equation}
		By applying \eqref{Le Scaphandre et le Papillon}, \eqref{Marriage Story} and Hölder's inequality, we infer that 
		\small{
		\begin{align}
			|\widehat{u}(t, \xi)| &\leq e^{MT} \left[|\widehat{g}(\xi)| e^{LT}  + \ds \int_{0}^{t} |\widehat{f}(s, \xi)| e^{L(t-s)}  ds \right] \leq e^{(L+M)T} \left[|\widehat{g}(\xi)|   + \ds \int_{0}^{t} |\widehat{f}(s, \xi)|   ds \right] \nonumber \\
			&  \leq C \left[|\widehat{g}(\xi)| + \left(\ds \int_{0}^{T} |\widehat{f}(s, \xi)|^{2}  ds \right)^{1/2} \right],  \ \ \forall t \in [0, T], \ \forall \xi \in \Z^{N}, \label{Mission: Impossible – Ghost Protocol}
		\end{align}
		}
		\normalsize
		for some $C > 0$. By Lemma \ref{American Gangster}, $u \in C\left([0, T]; H^{r - 2}(\T^N) \right)$. 
		
		We claim that $u$ is actually  an element of  $C\left([0, T]; H^{r}(\T^N) \right)$, a fact that will be proved through an  approximation argument.  Fix  $f \in C([0,T];H^{r}(\T^{N}))$, $g \in H^{r}(\T^{r})$ and let $\{f_j\}_{j \in \N} \subset C([0,T];C^{\infty}(\T^{N}))$, $\{g_j\}_{j \in \N} \in C^{\infty}(\T^{N})$ be sequences such that
		\begin{equation*}
		\lim_{j \to \infty} f_j = f \,\, \text{in} \,\, C([0,T];H^{r}(\T^{n})), \quad \lim_{j \to \infty} g_j = g \,\, \text{in} \,\, H^{r}(\T^{N}).
		\end{equation*} 
		It follows from \eqref{Mission: Impossible – Ghost Protocol} and Proposition \ref{The Graduate} that the Cauchy problem \eqref{Schindler's List} with initial data $f_j$ and $g_j$ admits a unique solution $u_{j} \in C([0,T]; C^{\infty}(\T^{N}))$.

		Hence it is a consequence  of Tonelli's theorem the existence of  ${C}_{1} > 0$ such that, for any $\sigma \in \R$,
		\small{
		\begin{equation*}
		\|u_{j}(t)\|^{2}_{H^{\sigma}(\T^{N})} \leq C_{1} \left( \|g_j\|^{2}_{H^{\sigma}(\T^{N})} + \int_{0}^{T} \|f_{j}(s)\|^{2}_{H^{\sigma}(\T^{N})} ds \right), \quad t \in [0,T].
		\end{equation*}
		}
		\normalsize
		Since $u_j - u_{k}$ is the solution of the problem \eqref{Schindler's List}, with data $f_j - f_k$ and $g_j - g_k$, we obtain
		\small{
		\begin{equation*}
		\left\| (u_{j}-u_k)(t)\right\|^{2}_{H^{\sigma}(\T^{N})} \leq C_{1} \left( \left\|g_j-g_k\right\|^{2}_{H^{\sigma}(\T^{N})} + \int_{0}^{T} \left\| (f_{j} -f_{k})(s)\right\|^{2}_{H^{\sigma}(\T^{N})} ds \right), \ \ \forall t \in [0,T].
		\end{equation*}
		}
		\normalsize
		In particular, $\{u_j\}_{j \in \N}$ is a Cauchy sequence in $C([0,T]; H^{r}(\T^{N}))$; let $v$ be the limit of $\{u_j\}_{j \in \N}$. Then $v \in C([0,T]; H^{r}(\T^{N}))$ and solves \eqref{Schindler's List}. From the uniqueness of the solutions one concludes that $v = u$, which finalizes the proof.
	\end{proof}

	\begin{Cor} \label{Lincoln}
		Consider $P$ the operator described in \eqref{Her} and suppose that 
		\begin{equation*}
			\Im (a_{2})(t) = c_{2}(t) < 0,  \  \text{for every  $t \in [0, T]$}.
		\end{equation*}
		In this case, we obtain the following properties:
		\begin{itemize}[leftmargin=*]
			\item Given $r \in \R$,  $f \in C \left([0,T]; H^{r}(\T^N) \right)$ and $g \in H^{r}(\T^N)$, there exists a unique $u \in C \left([0,T]; H^{r}(\T^N) \right) \cap C^{1} \left([0,T]; H^{r-2}(\T^N) \right) $ which solves \eqref{Schindler's List}.
			\item For every $f \in C \left([0,T]; C^{\infty}(\T^N) \right)$ and $g \in C^{\infty}(\T^N)$, there exists a unique $u \in C^{1} \left([0,T]; C^{\infty}(\T^N) \right)$ which solves \eqref{Schindler's List}.
			\item Given $s \geq 1$,  $f \in C \left([0,T]; G^{s}(\T^N) \right)$ and $g \in G^{s}(\T^N)$, there exists a unique $u \in C^{1} \left([0,T]; G^{s}(\T^N) \right) $ which solves \eqref{Schindler's List}.
		\end{itemize}
	\end{Cor}
	
\begin{proof}
It is a direct consequence of Theorem \ref{Rear Window}, Proposition \ref{The Graduate} and Remark \ref{The Descendants} 
\end{proof}
	
	Similarly to previous subsection, we exhibit a class of operators where $\Im (a_{2})$ is not necessarily strictly negative but we still  obtain well-posedness in the Sobolev setting.

	\begin{Teo} \label{First Man}
		Let $P$ be the operator defined in \eqref{Her} and assume that
		\begin{equation*}
			\Im (a_{2})(t) = c_{2}(t) \leq 0  \ \text{and} \ \Im (a_{1, j})(t) = c_{1, j}(t) \equiv 0, \  \text {for every $j \in \left\{1, 2, \ldots, N \right\}$ and  $t \in [0, T]$}.
		\end{equation*}
		Then \eqref{Schindler's List} is well-posed in $H^{r}$.  
	\end{Teo}
	
	\begin{proof}
		By proceeding just like it was done in order to obtain \eqref{The Untouchables} we have
		\small{
		\begin{equation*}
			|\widehat{u}(t, \xi)| \leq |\widehat{g}(\xi)| \exp \left\{ C_{2}(t) |\xi|^{2} +  C_{0}(t) \right\}  +  \ds \int_{0}^{t} |\widehat{f}(s, \xi)|   \exp \left\{\left(C_{2}(t) - C_{2}(s) \right) |\xi|^{2} + \left(C_{0}(t) - C_{0}(s) \right)  \right\} ds.
		\end{equation*}
		}
		\normalsize
		Take $M > 0$ such that $|c_{0}(t)| \leq M, \ \forall t \in [0, T]$. Then
		\begin{equation*}
			|\widehat{u}(t, \xi)| \leq e^{MT} \left[ |\widehat{g}(\xi)| e^{C_{2}(t) |\xi|^{2}} +     \ds \int_{0}^{t} |\widehat{f}(s, \xi)|  e^{\left(C_{2}(t) - C_{2}(s) \right) |\xi|^{2}} ds \right]. 
		\end{equation*}
		Since $C_{2}$ is always non-positive and non-increasing, it follows that
		\small{
		\begin{equation*}
			|\widehat{u}(t, \xi)| \leq e^{MT} \left[ |\widehat{g}(\xi)| +  \ds \int_{0}^{t} |\widehat{f}(s, \xi)| ds \right], \ \ \ \forall t \in [0, T], \ \ \forall \xi \in \Z^N.
		\end{equation*}
		}
		\normalsize
		Finally it suffices proceed just like in the end of the proof of Theorem \ref{Rear Window}.
	\end{proof}
	
	\begin{Cor} \label{Bourne Identity} 
		Consider $P$ the operator described in \eqref{Her} and suppose that 
		\begin{equation*}
			\Im (a_{2})(t) = c_{2}(t) \leq 0  \ \text{and} \ \Im (a_{1, j})(t) = c_{1, j}(t) \equiv 0, \  \text {for every $j \in \left\{1, 2, \ldots, N \right\}$ and  $t \in [0, T]$}.
		\end{equation*}
		Then the following statements hold true. 
		
		\begin{itemize}[leftmargin=*]
		\item Given $r \in \R$,  $f \in C \left([0,T]; H^{r}(\T^N) \right)$ and $g \in H^{r}(\T^N)$, there exists a unique $u \in C \left([0,T]; H^{r}(\T^N) \right) \cap C^{1} \left([0,T]; H^{r-2}(\T^N) \right) $ which solves \eqref{Schindler's List}.
		\item For every $f \in C \left([0,T]; C^{\infty}(\T^N) \right)$ and $g \in C^{\infty}(\T^N)$, there exists a unique $u \in C^{1} \left([0,T]; C^{\infty}(\T^N) \right)$ that  solves \eqref{Schindler's List}.
		\item Given $s \geq 1$,  $f \in C \left([0,T]; G^{s}(\T^N) \right)$ and $g \in G^{s}(\T^N)$, there exists a unique $u \in C^{1} \left([0,T]; G^{s}(\T^N) \right) $ which solves \eqref{Schindler's List}.
	\end{itemize}
	\end{Cor}

	\section{A class of degenerate operators} \label{The Dark Knight} 
	
	Let us make a brief analysis of some results proved in last sections; Corollary \ref{The Seventh Seal} shows that one only has to deal with problems where $\Im (a_{2}) = c_{2} \leq 0$, whilst  Theorem \ref{Rear Window} covers case $c_{2} < 0$.  Thus the only situations left are those where $c_{2}$ vanishes at some point. Taking a closer look at Theorems \ref{La La Land} and \ref{First Man} allows us to  conjecture about how the \emph{comparison of order of vanishing} between $c_{2}$ and the elements of $\left\{c_{1, 1}, c_{1, 2}, \ldots, c_{1, N}\right\}$ is connected to the well-posedness of \eqref{Schindler's List}. Roughly speaking, it seems that if $c_{2}$ vanishes to a much higher order than some $c_{1, j}$,  the problem is ill-posed. On the other hand, if each $c_{1, j}$ vanishes to a much higher order than $c_{2}$, \eqref{Schindler's List} is well-posed. 
	
	We intend to show in this section that such intuitive ideas are not that far from the truth, even though  the relations of well-posedness and ill-posedness are more intricate than the almost binary results obtained so far. In order to make things precise, we only consider situations where the imaginary part of the leading coefficient of $P$ given in \eqref{Her} has a finite number of zeros. Moreover the following hypotheses are assumed:

	\begin{enumerate} [wide,  labelindent=0pt]
		\item $\Im \ a_{2}(t) = c_{2}(t) \leq 0$, \emph{for every} $t \in [0, T]$.
		\item \emph{There exists a finite set}  $\left\{t_{1}, t_{2}, \ldots, t_{m}\right\} \subset [0, T]$, \emph{with} $t_{1} < t_{2} < \ldots < t_{m}$, \emph{such that}
		\begin{equation*} 
			\Im \ a_{2}(t) = c_{2}(t) = 0 \ \Leftrightarrow t \in \left\{t_{1}, t_{2}, \ldots, t_{m}\right\}.
		\end{equation*}
		\item \emph{There exists} $\delta > 0$ \emph{such that for every} $k \in \left\{1, 2, \ldots, m  \right\}$ \emph{we are able to write} 
		\begin{equation} \label{Psycho}
			\Im \ a_{2}(t) = c_{2}(t) = \alpha^{k}(t) \cdot |t - t_{k}|^{p_{k}}, \ \forall t \in [t_{k} -\delta, t_{k} + \delta] \cap [0, T],
		\end{equation}
		\emph{where} $p_{k} > 0$, $\alpha^{k}$ \emph{is bounded and there exists} $\varepsilon > 0$ \emph{such}  $\alpha^{k} \leq - \varepsilon$ \emph{on the same set}. 
		\item \emph{For each} $j \in \left\{1, 2, \ldots, N \right\}$ \emph{and} $k \in \left\{  1, \ldots, m \right\}$, \emph{we are able to write}
		\begin{equation} \label{The French Connection} 
			\Im \ a_{1, j}(t) = c_{1, j}(t) = \beta^{j, k}(t) \cdot |t - t_{k}|^{q_{j, k}}, \ \forall t \in [t_{k} -\delta, t_{k} + \delta ] \cap [0, T],
		\end{equation}  
		\emph{with} $q_{j, k} \geq 0$ \emph{and} $\beta^{j, k}$  \emph {a bounded function that does not  vanish  on the same set}. 	 
	\end{enumerate}

	\begin{Obs}
		Condition  \eqref{Psycho} implies that each zero of $c_{2}$ has positive and finite order; \eqref{The French Connection} is similar to \eqref{Psycho}, but it also includes the situation where $c_{1, j}$ does not vanish at some $t_{k}$. 
	\end{Obs}

	\begin{Obs}
		It is worth mentioning that not every continuous function satisfies \eqref{Psycho} or \eqref{The French Connection}. In addition to the classic example $f: [0, 1] \to \R_{-}; \ f(t) =  \left\{ \begin{array} {rl}
				- e^{-1/t^{2}}, & \text{if} \ t \neq 0 \\
				0, & \text{if} \ t = 0 
			\end{array} \right. $,
		of  a function that vanishes to infinite order at the origin, we also have $g: [0, 1/2] \to \R_{-}; \ g(t) =  \left\{ \begin{array} {rl}
			1/\log t, & \text{if} \ t \neq 0; \\
				0, & \text{if} \ t = 0, 
			\end{array} \right.$, 
	which is clearly continuous but $\ds \lim_{t \to 0^{+}} \left( \ds \frac{g(t)}{t^{r}} \right) = - \infty$, for each $r > 0$. However, note that this last example could not happen if the  function $g$ were for Hölder continuous for some $\alpha > 0$, for instance.  
	\end{Obs}
	
Using the continuity of $c_{2}$, we obtain (once again for the same $\delta > 0$) the following: 
\begin{enumerate} [wide,  labelindent=0pt, resume]
	
	\item \emph{There exists} $\omega > 0$ \emph{such that}
	\begin{equation} \label{Lost in Translation}
		t \in [0, T], \ \  |t - t_{k}| \geq \ds \frac{\delta}{2}, \ \forall k \in \left\{1, 2, \ldots, m \right\} \ \Rightarrow \ c_{2}(t) \leq - \omega.
	\end{equation} 		
	
\end{enumerate}

	\begin{Teo} \label{Tropa de Elite II} 
		Let $P$ be the operator defined in \eqref{Her} and set for each $k \in \left\{1, 2, \ldots, m \right\}$ 
		\begin{equation} \label{Skyfall}
			q_{k} = \ds \min \left\{q_{1, k}, \ldots, q_{N, k} \right\},
		\end{equation}
		with $q_{j,k}$ as in \eqref{The French Connection}. Then the following statements hold true: 
		\begin{enumerate} [label=\Roman*),  wide,  labelindent=0pt]
			\item \label{Vertigo} If $p_{k} \leq 2q_{k} +1$ for every $k \in \left\{1, 2, \ldots, m \right\}$, then \eqref{Schindler's List} is well-posed in $H^{r}$. 
			\item \label{Cidade de Deus} Suppose that $p_{k} > 2q_{k} +1$ for some $k \in \left\{1, 2, \ldots, m \right\}$.  Then \eqref{Schindler's List} is ill-posed in  $C^{\infty}$. Furthermore if we define
			\begin{equation}\label{Dunkirk}
				\mathscr{K} = \left\{1 \leq k \leq m; \  p_{k} >  2q_{k} +1 \right\},
			\end{equation}
			then \eqref{Schindler's List} is well-posed in $G^{\tau}$ if and only if
			\begin{equation} \label{Dune} 
				\tau <  \min_{k \in \mathscr{K}} \left( \frac{p_{k} - q_{k}}{p_{k} - 2q_{k} - 1} \right).
			\end{equation} 
		\end{enumerate}
	\end{Teo}
	
	\begin{proof}[\pmb{Proof of \ref{Vertigo}}]
		By proceeding just like we did in Subsection \ref{Gladiator}, it follows in particular from \eqref{The Untouchables} that
		\small{
			\begin{align*} 
				|\widehat{u}(t, \xi)|
				&\leq |\widehat{g}(\xi)| \exp  \left(C_{2}(t) |\xi|^{2} + \ds \sum_{j = 1}^{N} C_{1, j}(t) \xi_{j} +  C_{0}(t) \right)  \nonumber + \\
				&+ \ds \int_{0}^{t} |\widehat{f}(s, \xi)|   \exp  \left[\left(C_{2}(t) - C_{2}(s) \right) |\xi|^{2} + \ds \sum_{j = 1}^{N} \left(C_{1, j}(t) - C_{1, j}(s) \right) \xi_{j} +\left(C_{0}(t) - C_{0}(s) \right) \right]   ds,   
			\end{align*}
		}
		\normalsize
		for any $t \in [0,T]$ and $\xi \in \Z^N$. Take $M > 0$ such that 
		\begin{equation} \label{Spider Man 2}
			|c_{2}|, |c_{1, 1}|, \ldots, |c_{1, N}|, |c_{0}| \ \text{are all uniformly bounded by $\ds \frac{M}{N}$}.
		\end{equation}
		Then
		\small{
			\begin{equation} \label{Network} 
				\begin{split}
					|\widehat{u}(t, \xi)|
					&\leq e^{MT}|\widehat{g}(\xi)| \exp  \left(C_{2}(t) |\xi|^{2} + \ds \sum_{j = 1}^{N} C_{1, j}(t) \xi_{j}  \right)  + \\
					&+ \ds e^{MT} \int_{0}^{t} |\widehat{f}(s, \xi)|   \exp  \left[\left(C_{2}(t) - C_{2}(s) \right) |\xi|^{2} + \ds \sum_{j = 1}^{N} \left(C_{1, j}(t) - C_{1, j}(s) \right) \xi_{j}   \right]   ds.   
				\end{split}
			\end{equation}
		}
		\normalsize
		Our goal is to use \eqref{Network} to obtain an estimate similar to \eqref{Zero Dark Thirty}, with $\rho = 0$. Then it would be sufficient to repeat the arguments made in the end of the proof of Theorem \ref{Rear Window}.
	
		Take $t \in [0, T]$ and suppose in a first moment that
		\begin{equation} \label{Ghostbusters}
		\pmb{	|t - t_{k}| \geq \ds \frac{3 \delta}{4}, \ \ \ \forall k \in \left\{1, 2, \ldots, m \right\}}.
		\end{equation}
		If $t \leq \ds \frac{\delta}{4}$, by putting \eqref{Lost in Translation} and \eqref{Spider Man 2} into \eqref{Network} and applying triangular inequality,  we obtain   
		\begin{equation*}
			|\widehat{u}(t, \xi)| \leq e^{MT}  \left[|\widehat{g}(\xi)| \exp \left[(- \omega  |\xi|^{2} +  M|\xi|) t  \right] + \ds \int_{0}^{t} |\widehat{f}(s, \xi)|  \exp \left[ (-\omega |\xi|^{2} + M |\xi|) (t-s)  \right] ds  \right].
		\end{equation*}
		Since there exists $L > 0$ such that $-\omega |\xi|^{2} + M |\xi| \leq L$ for any $\xi \in \Z^{N}$, using Hölder's inequality we deduce that
		\small{
		\begin{align*}
			|\widehat{u}(t, \xi)| &\leq e^{MT}  \left[|\widehat{g}(\xi)|  e^{\frac{L \delta}{4}} +  \ds \int_{0}^{t} |\widehat{f}(s, \xi)| e^{\frac{L \delta}{4}} ds \right] \leq C \left[|\widehat{g}(\xi)| + \left(\ds \int_{0}^{T} |\widehat{f}(s, \xi)|^{2}  ds \right)^{1/2} \right],  \ \forall \xi \in \Z^{N},   
		\end{align*}
		}
		\normalsize
		for some constant $C > 0$ which only depends on $M, T$ and $\omega$.  
		
		When  $t > \ds \frac{\delta}{4}$, it follows from \eqref{Ghostbusters} that $\left|t - \ds \frac{\delta}{4} - t_{k}\right| \geq \left|t  - t_{k}\right| - \ds \frac{\delta}{4} \geq \ds \frac{\delta}{2}$, $\forall k \in \left\{1, 2, \ldots, m \right\}.$ Applying \eqref{Lost in Translation} and \eqref{Network}, we infer that
		\small{
		\begin{align*}
			|\widehat{u}(t, \xi)| &\leq e^{MT}   |\widehat{g}(\xi)| \exp \left[\left(C_{2}(t) - C_{2} \left(t - \delta/4\right) \right) |\xi|^{2} + MT |\xi|  \right]   \\
			&+ e^{MT}  \ds \int_{0}^{t - \delta/4} |\widehat{f}(s, \xi)|   \exp \left\{\left(C_{2}(t) - C_{2}(s) \right) |\xi|^{2} + MT |\xi|  \right\}  ds \\
			&+ e^{MT}  \ds \int_{t - \delta/4}^{t} |\widehat{f}(s, \xi)|   \exp \left\{\left(C_{2}(t) - C_{2}(s) \right) |\xi|^{2} + M(t-s) |\xi|  \right\}  ds \\
			&\leq  e^{MT}   |\widehat{g}(\xi)| \exp \left[\left(C_{2}(t) - C_{2} \left(t - \delta/4\right) \right) |\xi|^{2} + MT |\xi|  \right] \\
			&+ e^{MT}  \ds \int_{0}^{t - \delta/4} |\widehat{f}(s, \xi)|   \exp \left\{\left(C_{2}(t) - C_{2}(t - \delta/4) \right) |\xi|^{2} +  MT |\xi|  \right\}  ds  \\
			&+ e^{MT}  \ds \int_{t - \delta/4}^{t} |\widehat{f}(s, \xi)|   \exp \left\{- \omega \left(t-s \right) |\xi|^{2} + M (t-s) |\xi|  \right\}  ds \\
			&\leq e^{MT}   |\widehat{g}(\xi)| \exp \left\{ - \frac{\omega \delta}{4} |\xi|^{2} + MT |\xi|  \right\}  + e^{MT} \exp \left\{ - \frac{\omega \delta}{4} |\xi|^{2} + MT |\xi|  \right\} \ds \int_{0}^{t - \delta/4} |\widehat{f}(s, \xi)|   ds  \\
			&+ e^{MT} \ds \int_{t - \delta/4}^{t} |\widehat{f}(s, \xi)|   \exp \left\{ (- \omega  |\xi|^{2} + M|\xi|) \left(t-s \right) \right\}  ds.
		\end{align*}
		}
		\normalsize
	 	Let $L > 0$ such that 
		\begin{equation*}
			-\omega \xi^{2} + M|\xi| \leq L, \ \ \  - \ds \frac{\omega \delta}{4} |\xi|^{2} + MT |\xi| \leq L,  \ \ \  \forall \xi \in \Z^N. 
		\end{equation*}
		Then
		\small{
		\begin{align*}
			|\widehat{u}(t, \xi)| &\leq e^{MT + L} |\widehat{g}(\xi)| +  e^{MT + L} \ds \int_{0}^{t - \delta/4} |\widehat{f}(s, \xi)| ds +  e^{(L+M)T} \ds \int_{t - \delta/4}^{t} |\widehat{f}(s, \xi)|  ds \nonumber \\
			&\leq C \left[|\widehat{g}(\xi)| + \left(\ds \int_{0}^{T} |\widehat{f}(s, \xi)|^{2}  ds \right)^{1/2} \right],  \ \forall \xi \in \Z^{N}, 
		\end{align*}	
		}
		\normalsize
		for some constant $C > 0$ depending on $\omega, \delta, T$ and $M$, applying Hölder's inequality.

		Next we proceed to the situation where
		\begin{equation} \label{Casino Royale}
			\pmb{|t - t_{k}| <\ds \frac{3 \delta}{4}, \ \text{for some} \ k \in \left\{1, 2, \ldots, m \right\}}.
		\end{equation}
		Once again we split the proof into  two cases. 
		\begin{equation} \label{Birdman} 
			\text{\pmb{\emph{First case:} $t_{1} = 0$ and $t \in \left[0, \ds \frac{3 \delta}{4} \right]$}.}
		\end{equation}
		It follows from \eqref{Psycho}, \eqref{The French Connection} and \eqref{Network} that
		\small{
		\begin{align*}
			|\widehat{u}(t, \xi)|&\leq e^{MT}|\widehat{g}(\xi)| \exp  \left[\left(\ds \int_{0}^{t} \alpha^{1}(r) |r|^{p_{1}} dr \right) |\xi|^{2} + \ds \sum_{j = 1}^{N} \left(\ds \int_{0}^{t} \beta^{j, 1}(r) |r|^{q_{j, 1}} dr \right) \xi_{j}  \right]  + \\
			&+ \ds e^{MT} \int_{0}^{t} |\widehat{f}(s, \xi)|   \exp  \left[\left(\ds \int_{s}^{t} \alpha^{1}(r) |r|^{p_{1}} dr \right) |\xi|^{2} + \ds \sum_{j = 1}^{N} \left(\ds \int_{s}^{t} \beta^{j, 1}(r) |r|^{q_{j, 1}} dr \right) \xi_{j}   \right]   ds \\ 
			&\leq e^{MT}   |\widehat{g}(\xi)| \exp \left\{ -\varepsilon \left(\ds \int_{0}^{t}  r^{p_{1}} dr \right) |\xi|^{2} + \ds \sum_{j = 1}^{N} \left(\ds \int_{0}^{t} |\beta^{j, 1}(r)| r^{q_{j, 1}} dr \right) |\xi_{j}|  \right\}  +  \nonumber \\
			&+ e^{MT} \ds \int_{0}^{t} |\widehat{f}(s, \xi)|  \exp \left\{ - \varepsilon \left(\ds \int_{s}^{t}  r^{p_{1}} dr \right) |\xi|^{2} + \ds \sum_{j = 1}^{N} \left(\ds \int_{s}^{t} \left|\beta^{j, 1}(r)\right| r^{q_{j, 1}} dr \right) |\xi_{j}|  \right\}  ds.
		\end{align*}
		}
		\normalsize
		Take $\kappa > 0$ such that 
		\begin{equation} \label{Fight Club} 
			\left|\beta^{j, k}(r)\right| \leq \ds \frac{\kappa}{N}, \ \ \ \forall r \in [t_{k} -\delta, t_{k} + \delta ] \cap [0, T], \ \ \forall j \in \left\{1, 2, \ldots, N \right\}, \ \ \forall k \in \left\{1, 2, \ldots, m \right\}.
		\end{equation}
		Then
		\small{
		\begin{align*} 
			|\widehat{u}(t, \xi)| &\leq e^{MT}   |\widehat{g}(\xi)| \exp \left\{ -\varepsilon \left(\ds \int_{0}^{t}  r^{p_{1}} dr \right) |\xi|^{2} + \frac{\kappa}{N} \ds \sum_{j = 1}^{N} \left(\ds \int_{0}^{t}  r^{q_{j, 1}} dr \right) |\xi_{j}|  \right\}  +  \nonumber \\
			&+ e^{MT} \ds \int_{0}^{t} |\widehat{f}(s, \xi)|  \exp \left\{ - \varepsilon \left(\ds \int_{s}^{t}  r^{p_{1}} dr \right) |\xi|^{2} + \frac{\kappa}{N} \ds \sum_{j = 1}^{N} \left(\ds \int_{s}^{t}  r^{q_{j, 1}} dr \right) |\xi_{j}|  \right\}  ds.
		\end{align*}
		}
		\normalsize
		With no loss of generality one can assume that $\delta \leq 1$, which implies that $r^{q_{j, 1}} \leq r^{q_{1}}$  in the expression above, using the notation set in \eqref{Skyfall}.  Hence
		\small{ 
		\begin{equation} \label{1917}
		\begin{split}
			|\widehat{u}(t, \xi)| &\leq e^{MT}   |\widehat{g}(\xi)| \exp \left\{ -\varepsilon \left(\ds \int_{0}^{t}  r^{p_{1}} dr \right) |\xi|^{2} + \frac{\kappa}{N} \ds \sum_{j = 1}^{N} \left(\ds \int_{0}^{t}  r^{q_{1}} dr \right) |\xi_{j}|  \right\}  +  \\
			&+ e^{MT} \ds \int_{0}^{t} |\widehat{f}(s, \xi)|  \exp \left\{ - \varepsilon \left(\ds \int_{s}^{t}  r^{p_{1}} dr \right) |\xi|^{2} + \frac{\kappa}{N} \ds \sum_{j = 1}^{N} \left(\ds \int_{s}^{t}  r^{q_{1}} dr \right) |\xi_{j}|  \right\}  ds \\
			&= e^{MT}   |\widehat{g}(\xi)| \exp \left[ - \ds \frac{\varepsilon t^{p_{1} + 1}}{p_{1}+1}  |\xi|^{2} + \ds \frac{\kappa t^{q_{1} + 1}}{q_{1} + 1}   |\xi|  \right]  +    \\
			&+ e^{MT} \ds \int_{0}^{t} |\widehat{f}(s, \xi)| \exp \left\{ \ds \frac{\varepsilon (s^{p_{1} + 1} - t^{p_{1} + 1})}{p_{1}+1}  |\xi|^{2} + \ds \frac{\kappa (t^{q_{1} + 1} - s^{q_{1} + 1})}{q_{1} + 1}   |\xi|  \right\}  ds. 
			\end{split}
		\end{equation}
		}
		\normalsize
		
		By hypothesis, $p_{1} \leq 2q_{1} +1$; we split the first case in two parts, depending on the sign of $q_{1}-p_{1}$. 
		\begin{enumerate}[label=\Roman*),  wide,  labelindent=0pt]
			\item \pmb{$q_{1} < p_{1}$}. In this case $q_{1} < p_{1} \leq 2q_{1} +1$ and 
		\end{enumerate}
		\begin{equation} \label{Gran Torino}
			\nu_{1} := \ds \frac{2q_{1} - p_{1} + 1}{q_{1} - p_{1}} \leq 0. 
		\end{equation}
		So for any $\xi \in \Z^{N} \setminus \left\{0\right\}$ and $0 \leq s \leq t \leq \ds \frac{3 \delta}{4}$, we have
		\small{
			\begin{equation} \label{Phantom Thread} 
				\underbrace{\ds \frac{1}{|\xi|^{\nu_{1}}} \left[\ds \frac{\varepsilon (s^{p_{1} + 1} - t^{p_{1} + 1})}{(p_{1}+1)}  |\xi|^{2} + \ds \frac{\kappa (t^{q_{1} + 1} - s^{q_{1} + 1})}{q_{1} + 1}   |\xi| \right]}_{=:(\star)}   = \left(\ds \frac{\varepsilon (s^{p_{1} + 1} - t^{p_{1} + 1})}{(p_{1}+1)} \right) |\xi|^{2 - \nu_{1}} + \left(\ds \frac{\kappa (t^{q_{1} + 1} - s^{q_{1} + 1})}{q_{1} + 1} \right) |\xi|^{1- \nu_{1}}.  
			\end{equation}
		}
		\normalsize
		On the other hand, 
		\begin{equation} \label{The Martian}
			\begin{split}
				2 - \nu_{1} = \ds \frac{2q_{1} - 2p_{1} - 2q_{1} + p_{1} - 1}{q_{1} - p_{1}} = \ds \frac{p_{1} + 1}{p_{1} - q_{1}}; \ \ \ 1 - \nu_{1} = \ds \frac{q_{1} - p_{1} - 2q_{1} + p_{1} - 1}{q_{1} - p_{1}} = \ds \frac{q_{1}+1}{p_{1}-q_{1}}.
			\end{split}
		\end{equation}
		By associating \eqref{Phantom Thread} to \eqref{The Martian}, we infer that
		\small{
		\begin{equation} \label{Whiplash}
			\begin{split}
				(\star) &= \left[ \ds \frac{\varepsilon}{p_{1} + 1} \left(s|\xi|^{\frac{1}{p_{1} - q_{1}}}  \right)^{p_{1} + 1} - \ds \frac{\kappa}{q_{1} + 1} \left(s|\xi|^{\frac{1}{p_{1} - q_{1}}}  \right)^{q_{1} + 1} \right] - \left[ \ds \frac{\varepsilon}{p_{1} + 1} \left(t|\xi|^{\frac{1}{p_{1} - q_{1}}}  \right)^{p_{1} + 1} - \ds \frac{\kappa}{q_{1} + 1} \left(t|\xi|^{\frac{1}{p_{1} - q_{1}}}  \right)^{q_{1} + 1} \right].
			\end{split} 
		\end{equation}  
		}
		\normalsize
		
		Consider $f_{p_{1}, q_{1}, \varepsilon, \kappa}: \R_{+} \to \R$ given by 
		\begin{equation} \label{Mank}
			f_{p_{1}, q_{1}, \varepsilon, \kappa}(x) = - \ds \frac{\varepsilon x^{p_{1}+ 1}}{p_{1} + 1}  + \ds \frac{\kappa x^{q_{1}+ 1}}{q_{1} + 1}.
		\end{equation}
		Note that $f_{p_{1}, q_{1}, \varepsilon, \kappa}'(x) = - \varepsilon x^{p_{1}}  + \kappa x^{q_{1}}.$ Since  we are dealing with the case where $q_{1} < p_{1}$, $f_{p_{1}, q_{1}, \varepsilon, \kappa}$ is decreasing if $x > \left(\ds \frac{\kappa}{\varepsilon}\right)^{\frac{1}{p_{1} - q_{1}}}$. Thus there exists  $C_{p_{1}, q_{1}, \kappa, \varepsilon} > 0$ such that
		\begin{equation} \label{Cast Away}  
			0 \leq y \leq z \ \Rightarrow  \ f_{p_{1}, q_{1}, \kappa, \varepsilon}(z) - f_{p_{1}, q_{1}, \kappa, \varepsilon}(y) \leq C_{p_{1}, q_{1}, \kappa, \varepsilon}. 
		\end{equation}
		
		Observe that \eqref{Whiplash} can be rewritten as $(\star) = f_{p_{1}, q_{1}, \kappa, \varepsilon}(t|\xi|^{\frac{1}{p_{1} - q_{1}}}) - f_{p_{1}, q_{1}, \kappa, \varepsilon}(s|\xi|^{\frac{1}{p_{1} - q_{1}}}). $ By combining  \eqref{Phantom Thread} with \eqref{Whiplash} and \eqref{Cast Away}, one deduces that 
		\small{
		\begin{equation} \label{The Ilusionist}
			\ds \frac{\varepsilon (s^{p_{1} + 1} - t^{p_{1} + 1})}{(p_{1}+1)}  |\xi|^{2} + \ds \frac{\kappa (t^{q_{1} + 1} - s^{q_{1} + 1})}{q_{1} + 1}   |\xi| \leq C_{p_{1}, q_{1}, \kappa, \varepsilon} |\xi|^{\nu_{1}}, \ \ \forall \xi \in \Z^N \setminus \left\{0\right\}, \ \ 0 \leq s \leq t \leq \ds \frac{3 \delta}{4}. 
		\end{equation}
		}
		\normalsize
		But $\nu_{1}\leq 0$, which implies  
		\small{
		\begin{equation} \label{The Last Samurai}
			\ds \frac{\varepsilon (s^{p_{1} + 1} - t^{p_{1} + 1})}{(p_{1}+1)}  |\xi|^{2} + \ds \frac{\kappa (t^{q_{1} + 1} - s^{q_{1} + 1})}{q_{1} + 1}   |\xi|  \leq C_{p_{1}, q_{1}, \kappa, \varepsilon},  \ \ \forall \xi \in \Z^N, \ \ 0 \leq s \leq t \leq \ds \frac{3 \delta}{4}. 
		\end{equation}
		}
		\normalsize
		It follows  from \eqref{1917}, \eqref{The Last Samurai} and Hölder's inequality that 
		\small{
		\begin{align*}
			|\widehat{u}(t, \xi)| &\leq e^{MT} e^{C_{p_{1}, q_{1}, \kappa, \varepsilon}} \left(|\widehat{g}(\xi)| + \ds \int_{0}^{t} |\widehat{f}(s, \xi)|  ds \right) \leq C \left[|\widehat{g}(\xi)| + \left(\ds \int_{0}^{T} |\widehat{f}(s, \xi)|^{2}  ds \right)^{1/2} \right], \ \  \ \forall \xi \in \Z^{N}, 
		\end{align*}	
		}
		\normalsize
		for some constant $C > 0$ which depends on $M, T, p_{1}, q_{1}, \kappa$ and $\varepsilon$. 
		
		Next we deal with the case 
		\begin{enumerate}[resume, label=\Roman*),  wide,  labelindent=0pt]
			\item \pmb{$q_{1} \geq p_{1}$}. Then $0 \leq r \leq 1 \Rightarrow r^{q_{1}} \leq r^{p_{1}}$ and  there exists $B_{p_{1}, q_{1}} > 0$ for which $r^{q_{1}} \leq B_{p_{1}, q_{1}} r^{p_{1}}$, for every $0 \leq r \leq T$.  Thus
		\end{enumerate} 
		\small{
		\begin{equation} \label{Casino} 
			0 \leq s \leq t \leq T \ \Rightarrow \ds \int_{s}^{t} r^{q_{1}} dr \leq  B_{p_{1}, q_{1}} \ds \int_{s}^{t} r^{p_{1}} dr \ \ \Rightarrow  \ \ds \frac{(t^{q_{1} + 1} - s^{q_{1} + 1})}{q_{1} + 1} \leq B_{p_{1}, q_{1}} \ds \frac{(t^{p_{1} + 1} - s^{p_{1} + 1})}{p_{1} + 1}. 
		\end{equation}
		}
		\normalsize
		By associating \eqref{1917} to \eqref{Casino}, one infers that
		\small{
		\begin{align*}
			|\widehat{u}(t, \xi)| &\leq e^{MT}   |\widehat{g}(\xi)| \exp \left\{ - \ds \frac{\varepsilon t^{p_{1} + 1}}{p_{1}+1}  |\xi|^{2} + \ds \frac{\kappa B_{p_{1}, q_{1}}  t^{p_{1} + 1}}{p_{1} + 1}   |\xi|  \right\}  +   \\
			&+ e^{MT} \ds \int_{0}^{t} |\widehat{f}(s, \xi)|   \exp \left\{ \ds \frac{\varepsilon (s^{p_{1} + 1} - t^{p_{1} + 1})}{p_{1}+1}  |\xi|^{2} +  \ds \frac{\kappa B_{p_{1}, q_{1}}(t^{p_{1} + 1} - s^{p_{1} + 1})}{p_{1} + 1} |\xi| \right\}  ds  \\
			&\leq  e^{MT}  |\widehat{g}(\xi)| \exp \left\{t^{p_{1} + 1} \left(- \ds \frac{\varepsilon }{p_{1}+1}  |\xi|^{2} + \ds \frac{\kappa B_{p_{1}, q_{1}} }{p_{1} + 1}   |\xi| \right)  \right\}  +  \\
			&+ e^{MT} \ds \int_{0}^{t} |\widehat{f}(s, \xi)|  \exp \left\{(t^{p_{1} + 1} - s^{p_{1} + 1}) \left( - \ds \frac{\varepsilon}{p_{1}+1}  |\xi|^{2} +  \ds \frac{\kappa B_{p_{1}, q_{1}}}{p_{1} + 1} |\xi| \right) \right\}  ds.
		\end{align*}
		}
		\normalsize
		Since the expressions inside the exponentials are uniformly bounded, we conclude once again that
		\small{
		\begin{equation*} 
			|\widehat{u}(t, \xi)| \leq C \left[|\widehat{g}(\xi)| + \left(\ds \int_{0}^{T} |\widehat{f}(s, \xi)|^{2}  ds \right)^{1/2} \right],  \ \forall \xi \in \Z^{N},
		\end{equation*}
		}
		\normalsize
		for some constant $C$ depending on $M, T, p_{1}, q_{1}, \kappa$ and $\varepsilon$,  which finalizes the proof for \eqref{Birdman}.
		
		Now we proceed to the
		\begin{equation} \label{The Founder} 
			\text{\pmb{\emph{Second case:} $t_{1} \neq 0$ or $t \notin \left[0, \ds \frac{3 \delta}{4} \right]$}.}
		\end{equation}	
		Fix $t \in [0, T]$ e $k \in \left\{1, \ldots, m \right\}$ for which \eqref{Casino Royale} holds. By possibly reducing $\delta$, we may assume $\ds \frac{3 \delta}{2} \leq t_{k}.$
		Since $c_{2}\leq 0$, it follows from  \eqref{Spider Man 2} and \eqref{Network} that
		\small{
		\begin{align}
			|\widehat{u}(t, \xi)| &\leq e^{MT}  |\widehat{g}(\xi)| \exp \left\{ C_{2}\left(t_{k} - \ds \frac{3 \delta}{4}   \right) |\xi|^{2} + MT |\xi|  \right\} + \nonumber \\
			&+ e^{MT} \ds \int_{0}^{t_{k} - \delta} |\widehat{f}(s, \xi)|   \exp \left\{\left(C_{2}(t) - C_{2}(t_{k} - \delta) \right) |\xi|^{2} + MT|\xi|   \right\}  ds \nonumber + \\
			&+ e^{MT} \ds \int_{t_{k} - \delta}^{t} |\widehat{f}(s, \xi)|  \exp \left[\left(C_{2}(t) - C_{2}(s) \right) |\xi|^{2} + \ds \sum_{j = 1}^{N} \left(C_{1, j}(t) - C_{1, j}(s) \right) \xi_{j}   \right]  ds. \label{Spectre}
		\end{align}
		}
		\normalsize
		Note that, by \eqref{Lost in Translation} and \eqref{Casino Royale}, 
		\small{
		\begin{equation*}
			\begin{split}
				C_{2}\left(t_{k} - \ds \frac{3 \delta}{4}   \right) &= \ds \int_{0}^{t_{k} - 3\delta/4} c_{2}(r) dr \leq \ds \int_{t_{k} - \delta}^{t_{k} - 3\delta/4} c_{2}(r) dr \leq - \ds \frac{ \omega \delta }{4}, \\
				C_{2}(t) - C_{2}(t_{k} - \delta) &= \ds \int_{t_{k} - \delta}^{t} c_{2}(r) dr \leq  \ds \int_{t_{k} - \delta}^{t_{k} - 3\delta/4} c_{2}(r) dr \leq - \ds \frac{ \omega \delta }{4}.
			\end{split}
		\end{equation*}
		}
		\normalsize
		 Thus we can replace \eqref{Spectre} by 
		\small{
		\begin{align}
			|\widehat{u}(t, \xi)| &\leq e^{MT}  |\widehat{g}(\xi)| \exp \left\{ - \ds \frac{\omega \delta}{4} |\xi|^{2} + MT |\xi|  \right\} + e^{MT} \ds \int_{0}^{t_{k} - \delta} |\widehat{f}(s, \xi)|  \exp \left\{- \ds \frac{ \omega \delta }{4} |\xi|^{2} + MT|\xi|   \right\}  ds  \nonumber \\
			&+ e^{MT} \ds \int_{t_{k} - \delta}^{t} |\widehat{f}(s, \xi)|  \exp \left[\left(C_{2}(t) - C_{2}(s) \right) |\xi|^{2} + \ds \sum_{j = 1}^{N} \left(C_{1, j}(t) - C_{1, j}(s) \right) \xi_{j}   \right]  ds. \label{Dogville}
		\end{align}
		}
		\normalsize
		
		If we show that the exponential inside the third integral in of \eqref{Dogville} is uniformly bounded, by proceeding just like in previous situations we will be done. It follows from \eqref{Psycho}, \eqref{The French Connection} and \eqref{Casino Royale} that
		\footnotesize{
		\begin{equation} \label{Gone Girl}
				\underbrace{\left(C_{2}(t) - C_{2}(s) \right) |\xi|^{2} +  \sum_{j = 1}^{N} \left(C_{1, j}(t) - C_{1, j}(s) \right) \xi_{j}}_{(\dagger)} = \left(\ds \int_{s}^{t} \alpha^{k}(r) |r-t_{k}|^{p_{k}} dr \right) |\xi|^{2} + \ \sum_{j = 1}^{N} \left(\ds \int_{s}^{t} \beta^{j, k}(r) |r-t_{k}|^{q_{j,k}} dr \right) \xi_{j}. 
		\end{equation}
		}
		\normalsize
		Since one may assume $|r-t_{k}| \leq 1$, as a consequence of \eqref{Psycho}, \eqref{Skyfall} and \eqref{Fight Club} we obtain
		\small{
		\begin{align}
			(\dagger) &\leq   - \varepsilon \left(\ds \int_{s - t_{k}}^{t - t_{k}}  |r|^{p_{k}} dr \right) |\xi|^{2} + \kappa \left(\ds \int_{s - t_{k}}^{t - t_{k}}  |r|^{q_{k}} dr \right) |\xi|.  \label{Driving Miss Daisy}
		\end{align}
		}
		\normalsize
		By direct computations, one deduces
		\small{
		\begin{equation} \label{Unforgiven}
			\ds \int_{s - t_{k}}^{t - t_{k}}  |r|^{p_{k}} dr = \ds \frac{1}{p_{k} + 1}
			\begin{cases}
				\left( |t-t_{k}|^{p_{k} + 1} - |s-t_{k}|^{p_{k} + 1}\right),  &\quad\text{if $t \geq s \geq t_{k}$}; \\
				\left( |t-t_{k}|^{p_{k} + 1} + |s-t_{k}|^{p_{k} + 1}\right), & \quad\text{if $t \geq t_{k} \geq s$}; \\ 
				\left( |s-t_{k}|^{p_{k} + 1} - |t-t_{k}|^{p_{k} + 1}\right), & \quad\text{if $t_{k} \geq t \geq s$}.
			\end{cases}
		\end{equation} 
		}
		\normalsize
		An analogous expression is obtained when $p_{k}$ is replaced by $q_{k}$.  

		We split the estimate of $(\dagger)$ in \eqref{Gone Girl} into four parts:
		\begin{enumerate} [wide, labelindent=0pt]
			\item \pmb {$t \geq s \geq t_{k}$ and $q_{k} < p_{k} \leq  2q_{k} +1$}.  It follows from  \eqref{Driving Miss Daisy} and \eqref{Unforgiven} that 
		\begin{align*}
			(\dagger) &\leq  - \ds \frac{\varepsilon}{p_{k} + 1} \left(|t-t_{k}|^{p_{k} + 1} - |s-t_{k}|^{p_{k} + 1} \right) |\xi|^{2} + \ds \frac{\kappa}{q_{k} + 1} 	\left( |t-t_{k}|^{q_{k} + 1} - |s-t_{k}|^{q_{k} + 1}\right) |\xi|.
		\end{align*}	
		Now the proof is quite similar to one made with beginning at \eqref{Gran Torino}. Let
		\begin{equation} \label{Mystic River}
			\nu_{k} = \ds \frac{2q_{k} - p_{k} + 1}{q_{k} - p_{k}} < 0. 
		\end{equation} 
		We have an identity similar to \eqref{The Martian} and, by proceeding analogously, one concludes that
		\small{
		\begin{align*}
			\ds \frac{(\dagger)}{|\xi|^{\nu_{k}}} &\leq \left[ \ds \frac{\varepsilon}{p_{k} + 1} \left(|s-t_{k}| |\xi|^{\frac{1}{p_{k} - q_{k}}}\right)^{p_{k} + 1} - \ds \frac{\kappa}{q_{k} + 1} 	\left( |s-t_{k}| |\xi|^{\frac{1}{p_{k} - q_{k}}} \right)^{q_{k} + 1} \right] - \nonumber \\
			&-  \left[ \ds \frac{\varepsilon}{p_{k} + 1} \left(|t-t_{k}| |\xi|^{\frac{1}{p_{k} - q_{k}}}\right)^{p_{k} + 1} - \ds \frac{\kappa}{q_{k} + 1} 	\left( |t-t_{k}| |\xi|^{\frac{1}{p_{k} - q_{k}}} \right)^{q_{k} + 1} \right]. 
		\end{align*}
		}
		\normalsize
		Since $|t-t_{k}| \geq   |s - t_{k}|$, it suffices to repeat the arguments  with starting point at \eqref{Whiplash}. 
		\item \pmb{$t \geq t_{k} \geq s$ and $q_{k} < p_{k} \leq 2q_{k} + 1$}. By applying \eqref{Driving Miss Daisy} and \eqref{Unforgiven}, one gets
		\small{
		\begin{align*}
			(\dagger) &\leq  - \ds \frac{\varepsilon}{p_{k} + 1} \left(|t-t_{k}|^{p_{k} + 1} + |s-t_{k}|^{p_{k} + 1} \right) |\xi|^{2} + \ds \frac{\kappa}{q_{k} + 1} 	\left( |t-t_{k}|^{q_{k} + 1} + |s-t_{k}|^{q_{k} + 1}\right) |\xi|.
		\end{align*}
		}
		\normalsize
		Thus, by taking $\nu_{k}$ as in \eqref{Mystic River},
		\small{
		\begin{align*}
			\ds \frac{(\dagger)}{|\xi|^{\nu_{k}}} &\leq \left[- \ds \frac{\varepsilon}{p_{k} + 1} \left(|s-t_{k}| |\xi|^{\frac{1}{p_{k} - q_{k}}}\right)^{p_{k} + 1} + \ds \frac{\kappa}{q_{k} + 1} 	\left( |s-t_{k}| |\xi|^{\frac{1}{p_{k} - q_{k}}} \right)^{q_{k} + 1} \right] + \\
			&+  \left[- \ds \frac{\varepsilon}{p_{k} + 1} \left(|t-t_{k}| |\xi|^{\frac{1}{p_{k} - q_{k}}}\right)^{p_{k} + 1} + \ds \frac{\kappa}{q_{k} + 1} 	\left( |t-t_{k}| |\xi|^{\frac{1}{p_{k} - q_{k}}} \right)^{q_{k} + 1} \right].
		\end{align*}
		}
		\normalsize
		Using  notation set in \eqref{Mank}, we rewrite the inequality above as
		\small{
		\begin{equation*}
			\ds \frac{(\dagger)}{|\xi|^{\nu_{k}}} \leq f_{p_{k}, q_{k}, \varepsilon, \kappa}\left(|s-t_{k}| |\xi|^{\frac{1}{p_{k} - q_{k}}}\right) +  f_{p_{k}, q_{k}, \varepsilon, \kappa}\left(|t-t_{k}| |\xi|^{\frac{1}{p_{k} - q_{k}}}\right).
		\end{equation*}
		}
		\normalsize
		Since $f_{p_{k}, q_{k}, \varepsilon, \kappa}$ is decreasing  when $x > \left(\ds \frac{\kappa}{\varepsilon}\right)^{\frac{1}{p_{k} - q_{k}}}$,  it is still possible to obtain an inequality similar to \eqref{Cast Away} for the sum above and repeat the arguments made with starting point at \eqref{Whiplash}.
		\item \pmb{$t_{k} \geq t \geq s$ and $q_{k} < p_{k} \leq 2q_{k} + 1$.}  Using \eqref{Driving Miss Daisy} and \eqref{Unforgiven} we have
		\small{
		\begin{align*}
			(\dagger) &\leq  - \ds \frac{\varepsilon}{p_{k} + 1} \left(|s-t_{k}|^{p_{k} + 1} - |t-t_{k}|^{p_{k} + 1} \right) |\xi|^{2} + \ds \frac{\kappa}{q_{k} + 1} 	\left( |s-t_{k}|^{q_{k} + 1} - |t-t_{k}|^{q_{k} + 1}\right) |\xi|.
		\end{align*}
		}
		\normalsize
		Then
		\small{
		\begin{align*}
			\ds \frac{(\dagger)}{|\xi|^{\nu_{k}}} &\leq \left[ \ds \frac{\varepsilon}{p_{k} + 1} \left(|t-t_{k}| |\xi|^{\frac{1}{p_{k} - q_{k}}}\right)^{p_{k} + 1} - \ds \frac{\kappa}{q_{k} + 1} 	\left( |t-t_{k}| |\xi|^{\frac{1}{p_{k} - q_{k}}} \right)^{q_{k} + 1} \right] - \\
			&-  \left[ \ds \frac{\varepsilon}{p_{k} + 1} \left(|s-t_{k}| |\xi|^{\frac{1}{p_{k} - q_{k}}}\right)^{p_{k} + 1} - \ds \frac{\kappa}{q_{k} + 1} 	\left( |s-t_{k}| |\xi|^{\frac{1}{p_{k} - q_{k}}} \right)^{q_{k} + 1}\right],
		\end{align*}
		}
		\normalsize
		with $\nu_{k}$ given in \eqref{Mystic River}. By the fact that $t_{k} \geq t \geq s$, we deduce that $|s-t_{k}|  \geq  |t - t_{k}|$. Hence the proof in this situation is completely analogous to the first case.
		\item  \pmb{$q_{k} \geq p_{k}$}.  It is not necessary  to take into consideration here the sign of $(t_{k} - t)$ or $(t_{k} - s)$. Inded, since $q_{k} \geq p_{k}$, there exists $B_{p_{k}, q_{k}} > 0$ such that $|r| \leq T  \Rightarrow |r|^{q_{k}} \leq B_{p_{k}, q_{k}} |r|^{p_{k}}.$  One infers from \eqref{Driving Miss Daisy}
		\small{
		\begin{align*}
			(\dagger) &\leq  \left(\ds \int_{s - t_{k}}^{t - t_{k}}  |r|^{p_{k}} dr \right)  \left(- \varepsilon |\xi|^{2} + \kappa B_{p_{k}, q_{k}} |\xi| \right).
		\end{align*}
		}
		\normalsize
		Since both terms above are uniformly bounded, that closes the proof of \ref{Vertigo}.
	\end{enumerate} 	
\end{proof}

	\begin{proof}[\pmb{Proof of \ref{Cidade de Deus}}]
		
		Note the following consequence of the proof of \ref{Vertigo}:
		the behavior of $c_{2}$ and $c_{1, 1}, \ldots, c_{1, N}$ \emph{far} from the points $t_{1}, \ldots, t_{m}$ (to be more precise, if \eqref{Ghostbusters} holds) has no \emph{damage} to the well-posedness of \eqref{Schindler's List}. Furthermore, there exists no impact when $t$ is \emph{near} $t_{k}$ (being more accurate, if \eqref{Casino Royale} holds), \emph{provided that} $p_{k} \leq 2q_{k} + 1$.  Hence the only analysis left to be done is for neighborhoods of $t_{k}$, when $k$ is an element of the set $\mathscr{K}$ defined in \eqref{Dunkirk}. 
		
		We shall first prove the well-posedness of  \eqref{Schindler's List} in $G^{\tau}$ when $\mathscr{K} \neq \emptyset$ and \eqref{Dune} holds, splitting the proof in two parts, precisely described in \eqref{Birdman} and \eqref{The Founder}.  Similarly to \eqref{Gran Torino} and \eqref{Mystic River}, we set
		\begin{equation*}
			\nu_{k} = \ds \frac{2q_{k} - p_{k} + 1}{q_{k} - p_{k}} = \ds \frac{p_{k} - 2q_{k}   - 1}{p_{k} - q_{k}}.
		\end{equation*}
		However,  $\nu_{k}$ is \emph{strictly positive} if $k \in \mathscr{K}$.  Hence, by repeating the arguments applied from  \eqref{Phantom Thread} until \eqref{The Ilusionist}, as well as in each case of \eqref{The Founder}, we deduce that 
		\small{
		\begin{equation}  
			|\widehat{u}(t, \xi)| \leq e^{MT} e^{C_{p_{k}, q_{k}, \kappa, \varepsilon} |\xi|^{\nu_{k}}} \left(   |\widehat{g}(\xi)|  + \ds \int_{0}^{t} |\widehat{f}(s, \xi)|  ds \right),  \ \   |t - t_{k}| <\ds \frac{3 \delta}{4}, \ \forall \xi \in \Z^{N}.
		\end{equation}
		}
		\normalsize
		Let  $\nu := \ds \max_{k \in \mathscr{K}} \left\{\nu_{k} \right\}$; since we have a better estimate for the rest of the subsets of $[0,T]$, 
		\small{
		\begin{equation}\label{Amour}
			|\widehat{u}(t, \xi)| \leq C_{1} e^{ C_{2}|\xi|^{\nu}} \left(   |\widehat{g}(\xi)|  + \ds \int_{0}^{t} |\widehat{f}(s, \xi)|  ds \right),  \ \   \forall t \in [0, T], \ \ \forall \xi \in \Z^{N},
		\end{equation}
		}
		\normalsize 
		for some $C_{1}, C_{2} > 0$. Take $f \in C \left([0, T], G^{\tau}(\T^N) \right)$ and $g \in G^{\tau}(\T^{N})$; there exist $C_{3}, \sigma > 0$ such that 
		\small{
		\begin{equation} \label{Tropa de Elite} 
			|\widehat{f}(t, \xi)| \leq C_{3} e^{- \sigma |\xi|^{1/\tau}}, \ \ \ \ |\widehat{g}(\xi)| \leq C_{3} e^{- \sigma |\xi|^{1/\tau}}, \ \ \ \ \forall t \in [0, T], \ \forall \xi \in \Z^{N}.
		\end{equation}
		}
		\normalsize
		Associating \eqref{Amour} to \eqref{Tropa de Elite}, one infers, for some constant $C_{4} > 0$, that
		\small{
		\begin{equation*}  
			|\widehat{u}(t, \xi)| \leq C_{4} \exp \left(C_{2}|\xi|^{\nu}  - \sigma |\xi|^{1/\tau}  \right),  \ \ \ \ \forall t \in [0, T], \ \forall \xi \in \Z^{N},
		\end{equation*}
		}
		\normalsize
		
		 By \eqref{Dune} we have $\tau < \ds \min_{k \in \mathscr{K}} \left( \frac{p_{k} - q_{k}}{p_{k} - 2q_{k} - 1} \right)  \Rightarrow \ds \max_{k \in \mathscr{K}} \left( \frac{p_{k} - 2q_{k} - 1}{p_{k} - q_{k}} \right) < \ds \frac{1}{\tau}  \Rightarrow \nu < \ds \frac{1}{\tau}.$
		Thus there exists $C_{5} > 0$ for which
		\begin{equation*} 
			|\widehat{u}(t, \xi)| \leq C_{5} e^{- \frac{\sigma}{2} |\xi|^{1/\tau}},  \ \ \ \ \forall t \in [0, T], \ \forall \xi \in \Z^{N}.
		\end{equation*} 
		Our statement is then a consequence of Lemma \eqref{American Gangster}.   
		
		It remains to verify that if 
		\begin{equation} \label{Confessions of a Dangerous Mind}
			\varrho = \min_{k \in \mathscr{K}} \left( \frac{p_{k} - q_{k}}{p_{k} - 2q_{k} - 1} \right),
		\end{equation} 
		\eqref{Schindler's List} is ill-posed in $G^{\varrho}$. It follows from Proposition \ref{The Graduate} that the same problem is also ill-posed in $G^{s}$, for any $s > \varrho$, in $C^{\infty}$ and in $H^{r}$.  By Theorem \ref{Cries and Whispers}, it suffices to prove the result for the particular case where $\Re \ a_{2} \equiv \Re \ a_{1,1} \equiv \cdots  \equiv \Re \ a_{1,N}  \equiv a_{0} \equiv 0.$ In that situation, the Fourier coefficients of $u$ are  described in \eqref{Forrest Gump}.
		
		Denote by $k_{0}$ the element in   $\left\{1, 2, \ldots, m\right\}$ that satifies
		\begin{equation} \label{Notting Hill}
		\varrho =	\min_{k \in \mathscr{K}} \left(\frac{p_{k} - q_{k}}{p_{k} - 2q_{k} - 1} \right)	= \left( \frac{p_{k_0} - q_{k_0}}{p_{k_0} - 2q_{k_0} - 1} \right)
		\end{equation}
		and by $j_{0}$ the element in $\left\{1, 2, \ldots, N \right\}$ such that
		\begin{equation} \label{Good Will Hunting}
			q_{k_{0}} = \ds \min \left\{q_{1, k_{0}}, \ldots, q_{N, k_{0}} \right\} =  q_{j_{0}, k_{0}}.
		\end{equation}
		
		Consider in a first moment $\pmb{t_{k_0} \neq 0}$; in this case we may assume $t_{k_{0}} \geq \delta$. Let 
		\begin{equation*}
			\text{$g \equiv 0$ and $f(t, x) = -i \ds \sum_{\eta \in \Z} e^{-\vartheta |\eta|^{1/\varrho}}  e^{i x \cdot \eta^{j_{0}}}$},
		\end{equation*}
		where $ \eta^{j_{0}}$  denotes the vector $(0, \ldots, 0, \underbrace{\eta}_{\text{$j_{0}$-th term}}, 0, \ldots, 0)$ and $\vartheta$ is a positive number which will be chosen later. Then $\widehat{u}(t, \xi) = 0$  when $\xi$ is not of the form $\eta^{j_{0}}$ for some $\eta \in \Z$ and 
		\small{
		\begin{equation*} 
			\widehat{u}(t, \eta^{j_{0}}) =  e^{-\vartheta |\eta|^{1/\varrho}}  \ds \int_{0}^{t} \exp \left\{ \left[\left(C_{2}(t) - C_{2}(s) \right) \eta^{2} + \left(C_{1, j_{0}}(t) - C_{1, j_{0}}(s) \right) \eta  \right]  \right\} ds, \ \ \ \forall t \in [0, T], \ \ \forall \eta \in \Z.
		\end{equation*}
		}
		\normalsize
		Since the term inside the integral is positive, for any $t \in [t_{k_0}- \delta, t_{k_0}+\delta] \cap [0, T]$, we have
		\small{
		\begin{align}
			\widehat{u}(t, \eta^{j_{0}}) &\geq e^{-\vartheta |\eta|^{1/\varrho}}  \ds \int_{t_{k_{0}} - \delta}^{t} \exp \left\{ \left[\left(C_{2}(t) - C_{2}(s) \right) \eta^{2} + \left(C_{1, j_{0}}(t) - C_{1, j_{0}}(s) \right) \eta  \right]  \right\} ds \nonumber \\
			&= e^{-\vartheta |\eta|^{1/\varrho}}  \ds \int_{t_{k_{0}} - \delta}^{t} \exp \left[\left(\ds \int_{s}^{t} \alpha^{k_{0}}(r) |r-t_{k_0}|^{p_{k_0}} dr \right) \eta^{2} + \left(\ds \int_{s}^{t} \beta^{j_{0}, k_{0}}(r) |r-t_{k_0}|^{q_{j_{0}, k_0}} dr \right) \eta \right] ds, \label{Kramer vs Kramer} 
		\end{align}
		}
		\normalsize
		applying \eqref{Psycho} and \eqref{The French Connection}. Because $\beta^{j_{0}, k_{0}}$ does not vanish, the function is either strictly positive or negative in the interval $[t_{k_0} - \delta, t_{k_{0}}]$. In order to simplify the proof (the other case is analogous), we assume it is positive.  On the other hand, $\alpha^{k_{0}}$ is bounded; thus there exists $\gamma, \Gamma > 0$  for which ,
		\begin{equation*} 
		\beta^{j_{0},k_{0}} (r) \geq \gamma,  \ \ \ \	\alpha^{k_{0}} (r) \geq - \Gamma,  \ \ \ \forall r \in [t_{k_0} - \delta, t_{k_{0}}]. 
		\end{equation*} 
		By associating the inequalities above to \eqref{Good Will Hunting} and \eqref{Kramer vs Kramer}, we obtain for any $t \in [t_{k_0}- \delta, t_{k_0}]$ and  $\eta \in \N$,
		\small{
		\begin{equation} \label{Manchester by the Sea}
			\widehat{u}(t, \eta^{j_{0}})  \geq e^{-\vartheta |\eta|^{1/\varrho}}  \ds \int_{t_{k_{0}} - \delta}^{t} \exp \left[- \Gamma \left(\ds \int_{s}^{t}  |r-t_{k_0}|^{p_{k_0}} dr \right) \eta^{2} + \gamma \left(\ds \int_{s}^{t}  |r-t_{k_0}|^{q_{k_0}} dr \right) \eta \right] ds,
		\end{equation}
		}
	\normalsize
		
		We shall now analyze the term inside the exponential, which will be denoted by
		\small{
		\begin{equation*}
			(\triangle) := - \Gamma \left(\ds \int_{s}^{t}  |r-t_{k_0}|^{p_{k_0}} dr \right) \eta^{2} + \gamma \left(\ds \int_{s}^{t}  |r-t_{k_0}|^{q_{k_0}} dr \right) \eta. 
		\end{equation*}
		}
		\normalsize
		It follows from \eqref{Unforgiven} that
		\begin{equation*} 
			(\triangle) =   - \Gamma \left(\ds \frac{\left( |s-t_{k_0}|^{p_{k_{0}} + 1} - |t-t_{k_0}|^{p_{k_{0}} + 1}\right)}{p_{k_{0}} + 1} \right) \eta^{2} + \gamma \left(\ds \frac{\left( |s-t_{k_0}|^{q_{k_{0}} + 1} - |t-t_{l_0}|^{q_{k_{0}} + 1}\right)}{q_{k_{0}} + 1} \right) \eta. 
		\end{equation*}
		By proceeding analogously to what was done from \eqref{Gran Torino} up to \eqref{Cast Away}, we obtain for $\eta \geq 0$ 
		\small{
		\begin{align}
			(\triangle) &= \left[- \ds \frac{\Gamma}{p_{k_{0}} + 1} \left(|s-t_{k_0}| \eta^{\frac{1}{p_{k_{0}} - q_{k_{0}}}} \right)^{p_{k_{0}} + 1}  + \ds \frac{\gamma}{q_{k_{0}} + 1} \left(|s-t_{k_0}| \eta^{\frac{1}{p_{k_{0}} - q_{ k_{0}}}} \right)^{q_{k_{0}} + 1}  \right] \eta^{\nu_{k_{0}}} \nonumber \\
			& - \left[- \ds \frac{\Gamma}{p_{k_{0}} + 1} \left(|t-t_{k_0}| n^{\frac{1}{p_{k_{0}} - q_{k_{0}}}} \right)^{p_{k_{0}} + 1}  + \ds \frac{\gamma}{q_{ k_{0}} + 1} \left(|t-t_{k_0}| \eta^{\frac{1}{p_{k_{0}} - q_{k_{0}}}} \right)^{q_{k_{0}} + 1}  \right] \eta^{\nu_{k_{0}}} \nonumber \\
			&= \eta^{\nu_{k_0}} \left[f_{p_{k_{0}}, q_{ k_{0}}, \Gamma, \gamma} \left(|s-t_{k_0}| \eta^{\frac{1}{p_{k_{0}} - q_{k_{0}}}} \right) - f_{p_{k_{0}}, q_{k_{0}}, \Gamma, \gamma} \left(|t-t_{k_0}| \eta^{\frac{1}{p_{k_{0}} - q_{k_{0}}}} \right) \right], \label{Batman Begins} 
		\end{align}
		}	
		\normalsize
		where
		\small{
		\begin{equation} \label{A Beautiful Mind} 
			\nu_{k_{0}} = \ds \frac{p_{k_{0}} - 2q_{k_{0}} - 1}{ p_{k_{0}} - q_{ k_{0}} } \ \ \ \text{and} \ \ \ f_{p_{k_{0}}, q_{k_{0}}, \Gamma, \gamma}(x) = - \ds \frac{\Gamma x^{p_{k_{0}}+ 1}}{p_{k_{0}} + 1}  + \ds \frac{\gamma x^{q_{k_{0}}+ 1}}{q_{k_{0}} + 1}.
		\end{equation}
		}
		\normalsize
		
		Combining \eqref{Manchester by the Sea} with \eqref{Batman Begins}, one concludes that for any  $t \in [t_{k_0}- \delta, t_{k_0}]$ and $\eta \in \N$,
		\small{
			\begin{equation} \label{Argo}  
				\widehat{u}(t, \eta^{j_{0}})  \geq e^{-\vartheta |\eta|^{1/\varrho}}  \ds \int_{t_{k_{0}} - \delta}^{t} \exp \left\{\eta^{\nu_{k_0}} \left[f_{p_{k_{0}}, q_{k_{0}}, \Gamma, \gamma} \left(|s-t_{k_0}| \eta^{\frac{1}{p_{k_{0}} - q_{ k_{0}}}} \right) - f_{p_{k_{0}}, q_{k_{0}}, \Gamma, \gamma} \left(|t-t_{k_0}| \eta^{\frac{1}{p_{k_{0}} - q_{ k_{0}}}} \right) \right] \right\} ds.
			\end{equation}
		}
		\normalsize
		 Since $p_{k_{0}} > q_{k_{0}}$ and $f_{p_{k_{0}}, q_{k_{0}}, \Gamma, \gamma}'(x) = - \Gamma x^{p_{k_{0}}}  + \gamma x^{q_{k_{0}}},$ we have 
		 \small{
		 \begin{equation*}
		 \text{$f_{p_{k_{0}}, q_{k_{0}}, \Gamma, \gamma}$ strictly positive and increasing in the interval $ \left(0, \left(\ds \frac{\gamma}{\Gamma} \right)^{\frac{1}{p_{k_0} - q_{k_{0}}}}  \right)$.}
		 \end{equation*} 
	 	 }
 	 	\normalsize
	 Take  $\ell \in \N$ such that $\ds \frac{1}{\ell}$ belongs to this interval and consider the following sequences:
		\begin{equation}  \label{Up in the Air}
			t_{k_{0}, n} = t_{k_{0}} - \ds \frac{1}{3\ell n^{\frac{1}{p_{k_{0}} - q_{k_{0}}}}}, \ \ \ \tau_{k_{0}, n} = t_{k_{0}} - \ds \frac{1}{\ell n^{\frac{1}{p_{k_{0}} - q_{k_{0}}}}}, \ \ \ \iota_{k_{0}, n}= t_{k_{0}} - \ds \frac{1}{2\ell n^{\frac{1}{p_{k_{0}} - q_{k_{0}}}}}, \ \ \forall n \in \N.
		\end{equation}

		Taking $t = t_{k_{0}, n}$ as in \eqref{Up in the Air} and putting it in \eqref{Argo}, one has
		\small{
			\begin{equation*} 
				\widehat{u}(t_{ k_{0}, n}, n^{j_{0}}) \geq e^{-\vartheta n^{1/\varrho}}  \ds \int_{t_{k_{0}} - \delta}^{t_{ k_{0}, n}} \exp \left\{ n^{\nu_{ k_{0}}} \left[f_{p_{k_{0}}, q_{k_{0}}, \Gamma, \gamma} \left(|s-t_{k_0}| n^{\frac{1}{p_{k_{0}} - q_{ k_{0}}}} \right) - f_{p_{k_{0}}, q_{ k_{0}}, \Gamma, \gamma} \left(\frac{1}{3\ell} \right) \right]  \right\} ds. 
			\end{equation*}
		} 
		\normalsize
		Since we may take $\ell$ large enough so the sequences  introduced in \eqref{Up in the Air} are contained in $[t_{k_0} - \delta, t_{k_0}]$, we deduce from last inequality that
		\small{
			\begin{equation} \label{Ocean's Eleven}
				\widehat{u}(t_{k_{0}, n}, n^{j_{0}}) \geq e^{-\vartheta n^{1/\varrho}}  \ds \int_{\tau_{k_{0}, n}}^{\iota_{k_{0}, n}} \exp \left\{ n^{\nu_{ k_{0}}} \left[f_{p_{k_{0}}, q_{k_{0}}, \Gamma, \gamma} \left(|s-t_{k_0}| n^{\frac{1}{p_{k_{0}} - q_{ k_{0}}}} \right) - f_{p_{k_{0}}, q_{k_{0}}, \Gamma, \gamma} \left(\frac{1}{3\ell} \right) \right]  \right\} ds. 	
			\end{equation}
		}
		\normalsize
		Observe that $\tau_{k_{0}, n} \leq s \leq  \iota_{k_{0}, n}$ implies 
		\begin{equation*}
			\ds \frac{1}{2\ell} \leq |s - t_{k_0}|n^{\frac{1}{p_{k_{0}} - q_{k_{0}}}} \leq \ds \frac{1}{\ell}.
		\end{equation*}
		Note that our choice of $\ell$ was made so in the interval above $f_{p_{k_{0}}, q_{k_{0}}, \Gamma, \gamma}$ is increasing, which allows us to conclude that
		\small{
		\begin{equation} \label{A Separation}
			\widehat{u}(t_{k_{0}, n}, n^{j_{0}}) \geq e^{-\vartheta n^{1/\varrho}}  \ds \int_{\tau_{k_{0}, n}}^{\iota_{ k_{0}, n}} \exp \left\{ n^{\nu_{k_{0}}} \left[f_{p_{k_{0}}, q_{k_{0}}, \Gamma, \gamma} \left(\ds \frac{1}{2 \ell} \right) - f_{p_{k_{0}}, q_{k_{0}}, \Gamma, \gamma} \left(\frac{1}{3\ell} \right) \right]  \right\} ds.
		\end{equation}
		}
		\normalsize
		Let $\varsigma = f_{p_{k_{0}}, q_{k_{0}}, \Gamma, \gamma} \left(\ds \frac{1}{2\ell} \right) - f_{p_{k_{0}}, q_{k_{0}}, \Gamma, \gamma} \left(\ds \frac{1}{3 \ell} \right) > 0$. By \eqref{Up in the Air} and \eqref{A Separation}, 
		\small{
		\begin{align}
			\widehat{u}(t_{ k_{0}, n}, n^{j_{0}}) &\geq  \exp \left(-\vartheta n^{1/\varrho} + \varsigma n^{\nu_{ k_{0}}} \right)  \left(\iota_{ k_{0}, n} - \tau_{ k_{0}, n} \right) =  \exp \left(-\vartheta n^{1/\varrho} + \varsigma n^{\nu_{k_{0}}} \right) \ds \frac{1}{2\ell n^{\frac{1}{p_{k_{0}} - q_{k_{0}}}}}, \ \ \forall n \in \N.  \label{Matrix} 
		\end{align}
		}
		\normalsize
		On the other hand, it follows from   \eqref{Confessions of a Dangerous Mind}, \eqref{Notting Hill}, \eqref{Good Will Hunting} and \eqref{A Beautiful Mind} that  
		\begin{equation} \label{My Left Foot}
			\nu_{ k_{0}} = \ds \frac{p_{k_{0}} - 2q_{k_{0}} - 1}{ p_{k_{0}} - q_{ k_{0}} } = \ds \frac{1}{\varrho}.	
		\end{equation}
		Therefore, if $\vartheta = \ds \frac{\varsigma}{2}$, we conclude from \eqref{Matrix} that 
		\begin{equation*}
			\widehat{u}(t_{k_{0}, n}, n^{j_{0}}) \geq \ds \frac{\exp \left( \frac{\varsigma}{2} n^{1/\varrho} \right)} {2\ell n^{\frac{1}{p_{k_{0}} - q_{k_{0}}}}} \ \Rightarrow \ \ds \lim_{n \to +\infty } \widehat{u}(t_{ k_{0}, n}, n^{j_{0}}) = + \infty, 
		\end{equation*}
		which proves the ill-posedness of \eqref{Schindler's List} in $G^{\varrho}$. 
		
		Finally, we analyze the case where $\pmb{t_{k_0} = t_{1} = 0}$; the main difference here lies on the choice of the right-hand side functions. Take
		\begin{equation*}
			\text{$f \equiv 0$ and $g(x) =  \ds \sum_{\eta \in \Z} e^{- \vartheta |\eta|^{1/\varrho}} e^{i x \cdot \eta^{j_{0}}}$}
		\end{equation*}
		for some $\vartheta > 0$ that will be defined later and $j_{0}$ the number for which $q_{1} = \ds \min \left\{q_{1, 1}, \ldots, q_{N, 1} \right\} =  q_{j_{0}, 1}$. By proceeding just like in last case, we deduce that
		\begin{equation*} 
			\widehat{u}(t, \eta^{j_{0}}) = e^{- \vartheta |\eta|^{1/\varrho}}  \exp  \left[C_{2}(t) \eta^{2} + C_{1, j_{0}}(t) \eta  \right], \ \ \ \forall t \in [0, T], \ \ \forall \eta \in \Z.
		\end{equation*}
		When $0 \leq t \leq \delta$, one has
		\small{
		\begin{align*}
			\widehat{u}(t, \eta^{j_{0}}) &=  e^{- \vartheta |\eta|^{1/\varrho}} \exp \left[ \left(\ds \int_{0}^{t} \alpha^{1}(r) r^{p_{1}} dr \right) \eta^{2} + \left(\ds \int_{0}^{t} \beta^{j_{0}, 1}(r) r^{q_{j_{0}, 1}} dr \right) \eta  \right].
		\end{align*}
		}
		\normalsize
		With the same hypotheses and inequalities applied in \eqref{Manchester by the Sea}, we obtain
		\small{
		\begin{align*}
			\widehat{u}(t, \eta^{j_{0}}) &\geq  e^{- \vartheta |\eta|^{1/\varrho}} \exp \left[ - \Gamma \left(\ds \int_{0}^{t}  r^{p_{1}} dr \right) \eta^{2} + \gamma \left(\ds \int_{0}^{t} r^{q_{j_{0}, 1}} dr \right) \eta  \right]   \\
			&\geq e^{- \vartheta |\eta|^{1/\varrho}} \exp \left[ - \ds \frac{\Gamma}{p_{1} + 1} t^{p_{1} +1}   \eta^{2} + \ds \frac{\gamma}{q_{j_{0}, 1} + 1}  t^{q_{j_{0}, 1} + 1}  \eta  \right].
		\end{align*}  	 
		}
		\normalsize
		With a similar argument to the one used in \eqref{Batman Begins}, one has for $\eta \in \N$
		\small{
		\begin{align*}
			\widehat{u}(t, \eta^{j_{0}}) &\geq   e^{- \vartheta |\eta|^{1/\varrho}} \exp \left\{ \left[ - \ds \frac{\Gamma}{p_{1} + 1} \left(t   \eta^{\frac{1}{p_{1} - q_{1}}} \right)^{p_{1} + 1} + \ds \frac{\gamma}{q_{1} + 1}  \left(t  n^{\frac{1}{p_{1} - q_{1}}} \right)^{q_{ 1} + 1}  \right] \eta^{\nu_{1}} \right\} \\
			&=  e^{- \vartheta |\eta|^{1/\varrho}} \exp \left[ \eta^{\nu_{1}}  f_{p_{1}, q_{1}, \Gamma, \gamma} \left(t \eta^{\frac{1}{p_{1} - q_{1}}} \right) \right].
		\end{align*}		
		}
		\normalsize
		Analogously,  $f_{p_{1}, q_{1}, \Gamma, \gamma}$  is strictly positive and increasing in the interval $ \left(0, \left(\ds \frac{\gamma}{\Gamma} \right)^{\frac{1}{p_{1} - q_{1}}}  \right)$. Take  $\ell \in \N$ such that $\ds \frac{1}{\ell}$ belongs to this interval and consider the sequence  $t_{1, n} = \ds \frac{1}{\ell n^{\frac{1}{p_{1} - q_{1}}}}$. Then
		\small{
		\begin{align*}
			\widehat{u}(t_{1, n}, n^{j_{0}}) &\geq  e^{- \vartheta n^{1/\varrho}} \exp \left[ n^{\nu_{1}} f_{p_{1}, q_{ 1}, \Gamma, \gamma} \left(\ds \frac{1}{\ell} \right) \right] \geq \exp \left[ \left(f_{p_{1}, q_{ }, \Gamma, \gamma} \left(\ds \frac{1}{\ell} \right) - \vartheta  \right) n^{\nu_{1}} \right],
		\end{align*}	
		}
		\normalsize
		since \eqref{My Left Foot} holds once again. 
		Therefore it suffices to take  $\vartheta =  f_{p_{1}, q_{ 1}, \Gamma, \gamma} \left(\ds \frac{1}{\ell} \right) /2$, which entails $ \ds \lim_{n} \widehat{u}(t_{1, n}, n^{j_{0}}) = + \infty$, finalizing the  proof of Theorem \ref{Tropa de Elite II}. 	
	\end{proof}

\begin{Cor} \label{The Bourne Ultimatum}
With the same hypotheses of Theorem \ref{Tropa de Elite II}, suppose that $p_{k} \leq 2q_{k} +1$ for each $k$ in $\left\{1, 2, \ldots, m \right\}$. In this  situation, we have the following properties:
\begin{itemize}[leftmargin=*]
	\item Given $r \in \R$,  $f \in C \left([0,T]; H^{r}(\T^N) \right)$ and $g \in H^{r}(\T^N)$, there exists a unique $u \in C \left([0,T]; H^{r}(\T^N) \right) \cap C^{1} \left([0,T]; H^{r-2}(\T^N) \right) $ which solves \eqref{Schindler's List}.
	\item For every $f \in C \left([0,T]; C^{\infty}(\T^N) \right)$ and $g \in C^{\infty}(\T^N)$, there exists a unique $u \in C^{1} \left([0,T]; C^{\infty}(\T^N) \right)$ that  solves \eqref{Schindler's List}.
	\item Given $s \geq 1$,  $f \in C \left([0,T]; G^{s}(\T^N) \right)$ and $g \in G^{s}(\T^N)$, there exists a unique $u \in C^{1} \left([0,T]; G^{s}(\T^N) \right) $ which solves \eqref{Schindler's List}.
\end{itemize}

\end{Cor}

\section{Final Remarks} \label{Ex Machina} 

In this section we present examples and facts that can be easily deduced from results proved previously and may be applied for other problems that were not precisely stated before.

\begin{Obs}
It is worth noting that Proposition \ref{Seven} and Theorem \ref{Cries and Whispers} can both be extended if we replace $Q(t, D_{x})$ in \eqref{Her} by any linear differential operator of any order $m$,  with the only difference being the loss of regularity, where  $2$ is replaced $m$, which once again will not matter in the $C^{\infty}$ and $G^{s}$ contexts. Therefore,  whenever one deals with operators in the form $D_{t} - Q(t, D_{x})$, only the imaginary part of $Q$ truly matters from a solvability standpoint.  Furthermore, the zero-order term can be neglected as well.
\end{Obs}

\begin{Obs}
In similar fashion, Lemma \ref{American Gangster} and Proposition \ref{The Graduate} can be obtained  if one replaces $Q(t, D_{x})$ in \eqref{Her} by any linear differential operator of any order $m$. In this case, the only difference is the first statement of the lemma, where $u$ becomes  an element of $C\left([0, T]; H^{r - m - \rho}(\T^N) \right)$. 
\end{Obs}

\begin{Obs}
When $N = 1$, any second-order operator $Q(t, D_{x})$ has the form described in \eqref{Her}. Furthermore, consider the case given by
\begin{equation*}  
	P(t, D_{x}, D_{t}) =  D_{t} + a_{3}(t) D_{x}^{3}+ a_{2}(t) D_{x}^{2} + a_{1}(t) D_{x} + a_{0}(t), \ \ \ \ t \in [0, T], \ x \in \T^N. 
\end{equation*}	
Proceeding just like in Section \ref{The Post}, we obtain a solution quite similar to the one described in \eqref{Sunset Boulevard}. By the argument used in last remark, one may assume that $Q$ only  has pure imaginary coefficients and no zero-order term. After all this machinery, following \eqref{Forrest Gump}, we deduce that
\small{
\begin{equation*}  
\begin{split}
\widehat{u}(t, \xi) &= \widehat{g}(\xi) \exp  \left[C_{3}(t) \xi^{3} + C_{2}(t) \xi^{2} + C_{1}(t) \xi  \right]  + \\
&+ \ds i \int_{0}^{t} \widehat{f}(s, \xi)  \exp  \left[\left(C_{3}(t) - C_{3}(s) \right) \xi^{3} + \left(C_{2}(t) - C_{2}(s) \right) \xi^{2} +   \left(C_{1}(t) - C_{1}(s) \right) \xi  \right]   ds,  \ \ \forall t \in [0, T], \ \ \forall \xi \in \Z.
\end{split}
\end{equation*}
}
\normalsize
Now note the following: the sign of $\xi^{3}$, in opposition to $\xi^{2}$, \emph{changes} when one alters the sign of $\xi$. Hence, if $c_{3}$ does not vanish at some point, we are able to proceed analogously to what was done in Section \ref{Memento} and prove ill-posedness in any setting (it only boils down to \emph{picking the correct direction}). Thus the only case where it is possible to conclude well-posedness is when $c_{3} \equiv 0$, which reverts back to the situation we dealt with throughout the paper.
\end{Obs}

\begin{Obs}
The intriguing consequence of Theorem \ref{Tropa de Elite II} that, with such hypotheses, one always has $C^{\omega}$ well-posedness, is due to the fact the operator $Q$ in \eqref{Her} has order $2$. In fact, consider the following fourth-order operator: 
\begin{equation*}
P(t, D_{x}, D_{t}) = D_{t} + i \left[ - 5t^{4} D_{x}^{4} + 3  t^{2} D_{x}^{3} +   D_{x}^{2} +  2  t D_{x}   \right], \ \ \ \ t \in [0, T], \ x \in \T.
\end{equation*}	
In this case, \eqref{Sunset Boulevard} becomes, for each $t \in [0, T]$ and $\xi \in \Z$,
\footnotesize{
\begin{equation*} 
	\begin{split}
		\widehat{u}(t, \xi) &= \widehat{g}(\xi) \exp \left\{- t^{5}\xi^{4} +  t^{3} \xi^{3} +  t \xi^{2} +  t^{2} \xi    \right\}  +  i \int_{0}^{t} \widehat{f}(s, \xi)  \exp  \left[- \left(t^{5} - s^{5} \right) \xi^{4} +  \left(t^{3} - s^{3} \right) \xi^{3} +  \left(t - s \right) \xi^{2} + \left(t^{2} - s^{2} \right) \xi  \right] 	   ds.  
	\end{split}	
\end{equation*}
}
\normalsize

Now take $f \equiv 0$ and $g = \ds \sum_{\eta \in \Z} e^{- |\eta|} e^{i x \eta}$; thus $f \in C(\left[0, T\right]; C^{\omega}(\T))$ and $g \in C^{\omega}(\T)$, and the expression above is turned into $\widehat{u}(t, \xi) = \exp \left\{- t^{5}\xi^{4} +  t^{3} \xi^{3} +  t \xi^{2} +  t^{2} \xi - |\xi|    \right\}$, for any $t \in [0, T]$ and $\xi \in \Z$. Consider the sequence $t_{n} = \ds \frac{1}{n^{4/5}}$. Then, for each $n \in \N$, we obtain
\begin{equation*}
\widehat{u}(t_{n}, n) = \exp \left\{- 1 +   n^{3/5} +  n^{6/5} +  n^{-3/5} - n   \right\} \ \ \Rightarrow \ \ \lim_{n  \to \infty} \widehat{u}(t_{n}, n) = +\infty. 
\end{equation*}
Thus in this case the Cauchy problem \eqref{Schindler's List} is not well-posed in any of the aforementioned spaces.  

\end{Obs}

\begin{Obs} \label{The Super Mario Bros. Movie}
It is also worth mentioning one cannot obtain an analogous result of Theorem \ref{Tropa de Elite II} when we consider functions that vanish to an infinite order at some point. Just to make things simples, we assume $N = 1$ and take operators given as in \eqref{Raiders of the Lost Ark} . 

Consider first $c_{2}(t) = - \ds \frac{2 e^{-1/t^{2}}}{t^{3}}$ and $c_{1}(t) = \ds \frac{2 e^{-1/t^{2}}}{t^{3}}$; in this case, applying formula \eqref{Forrest Gump}, we have $C_{2}(t) = - e^{-1/t^{2}}$, $C_{1}(t) = e^{-1/t^{2}}$ and
\small{
	\begin{equation*}
		\begin{split}
			\widehat{u}(t, \xi) &= \widehat{g}(\xi) \exp  \left[ e^{-1/t^{2}}  (\xi - |\xi|^{2})  \right]  + \ds i \int_{0}^{t} \widehat{f}(s, \xi)  \exp  \left[\left(e^{-1/s^{2}} - e^{-1/t^{2}}    \right) (\xi - |\xi|^{2}) \right]   ds,  \ \ \forall t \in [0, T], \ \ \forall \xi \in \Z,
		\end{split}
	\end{equation*}
}
\normalsize
which implies that 
\begin{equation*}
\left|\widehat{u}(t, \xi) \right| \leq |\widehat{g}(\xi)| + \ds  \int_{0}^{t} \left|\widehat{f}(s, \xi) \right| ds,  \ \ \forall t \in [0, T], \ \ \forall \xi \in \Z.
\end{equation*} 
Hence this is an example of operator that is well-posed in $H^{r}$ regarding the Cauchy Problem \eqref{Schindler's List}.

 On the other hand,  let $c_{2}(t) = -  \ds \frac{2 e^{-1/t^{2}}}{t^{3}}$ and $c_{1}(t) =  \ds \frac{2 e^{-1/5t^{2}}}{5 t^{3}}$, which means that $C_{2}(t) = - e^{-1/t^{2}}$ and $C_{1}(t) = e^{-1/5t^{2}}$. Consider $f \equiv 0$ and $g(x) = \ds \sum_{\eta \in \Z} e^{- |\eta|^{1/2}} e^{i x  \eta} $; similarly to the first case we have
\begin{align*} 
	\widehat{u}(t, \eta)  =  e^{-  |\eta|^{1/2}}  \exp  \left[-e^{-1/t^{2}} \eta^{2} + e^{-1/5t^{2}} \eta  \right],  \ \ \forall t \in [0, T], \ \ \forall \eta \in \Z. 
\end{align*}
Consider, for $n \in \N$ sufficiently large, $t_{n} = \ds \frac{1}{\sqrt{2 \log n}}$. Then $t_{n}^{2} = \ds \frac{1}{2 \log n} = \ds \frac{1}{ \log n^{2}} \ \Rightarrow \ - \ds\frac{1}{t_{n}^{2}} =  - \log n^{2}$, which implies that
\begin{equation*}
	\exp \left(- \ds\frac{1}{t_{n}^{2}} \right) = \ds \frac{1}{n^{2}} \ \ \text{and} \ \ \exp \left(- \ds\frac{1}{5 t_{n}^{2}} \right) =   \exp \left(- \ds \frac{ \log n^{2}}{5} \right) = \ds \frac{1}{n^{2/5}}.
\end{equation*}
Therefore, for $n \in \N$ sufficiently large, 
\begin{equation*}
	\widehat{u}(t_{n}, n) = e^{- n^{1/2}}  \exp  \left[- \ds \frac{1}{n^{2}} n^{2} + \ds \frac{1}{n^{2/5}} n  \right] = e^{n^{3/5} - n^{1/2} -1} \ \ \Rightarrow \ \  \widehat{u}(t_{n}, n) \to \infty \ \text{when $n \to \infty$}.
\end{equation*}
Hence  the Cauchy Problem \eqref{Schindler's List} is not well-posed in $G^{2}$, which shows our assertion. 
\end{Obs}

\end{document}